\documentclass[11pt,reqno]{amsart}

\usepackage{amsfonts,amssymb,amsbsy,amsmath,amsthm,dsfont}
\usepackage[usenames,dvipsnames]{color}
\usepackage{color}
\newtheorem{theorem}{Theorem}[section]
\newtheorem{lemma}[theorem]{Lemma}
\newtheorem{follow}[theorem]{Corollary}
\newtheorem{corl}[theorem]{Corollary}
\newtheorem{prop}[theorem]{Proposition}

\theoremstyle{definition}

\theoremstyle{remark}

\newtheorem{remark}[theorem]{\bf Remark}

\numberwithin{equation}{section}

\newcommand{\bel}{\begin{equation} \label}
\newcommand{\ee}{\end{equation}}
\newcommand{\rt}{{\mathbb R}^{3}}
\newcommand{\rd}{{\mathbb R}^{2}}
\newcommand{\re}{{\mathbb R}}
\newcommand{\bx}{{\bf x}}
\newcommand{\eps}{{\varepsilon}}
\newcommand{\cal}{\mathcal}

\newcommand{\dcot}{\dot{C}^{\infty}(\Omega_\theta)}

\newcommand{\B}{\mathcal{B}}

\newcommand{\E}{\mathcal{E}}
\newcommand{\h}{\rm H}

\newcommand{\J}{\mathcal{J}}
\newcommand{\K}{\mathcal{K}}
\newcommand{\el}{\mathcal{L}}
\newcommand{\M}{\mathcal{M}}
\newcommand{\pd}{\partial}
\renewcommand{\epsilon}{\varepsilon}
\newcommand{\id}{\mathds{1}}
\newcommand{\F}{\mathcal{F}}

\newcommand{\N}{\mathbb{N}}

\begin{document}
\title{Scattering in twisted waveguides}

\author{Philippe Briet}

\author {Hynek Kova\v{r}\'{\i}k}

\author {Georgi  Raikov}

\begin{abstract}

We  consider a twisted  quantum waveguide, i.e.,  a domain of the
form $\Omega_{\theta} : =  r_\theta \omega \times \re$ where
$\omega \subset \rd$ is a  bounded domain, and $r_\theta=
r_\theta(x_3) $ is  a rotation by the angle $\theta(x_3)$
depending on the longitudinal variable $x_3$. We investigate the
nature of the essential spectrum of the Dirichlet Laplacian
${\mathcal H}_\theta$, self-adjoint in ${\rm L}^2(\Omega_\theta)$,
and consider related scattering problems. First, we show that if
the derivative of the difference $\theta_1 - \theta_2$ decays fast
enough as $|x_3| \to \infty$, then the wave operators for the
operator pair $({\mathcal H}_{\theta_1}, {\mathcal H}_{\theta_2})$
exist and are complete. Further, we concentrate on appropriate
perturbations of constant twisting, i.e.  $\theta' = \beta -
\varepsilon$ with constant $\beta \in \re$, and $\varepsilon$
which  decays fast enough at infinity together with its first
derivative. In that case the unperturbed operator corresponding to
$\varepsilon$ is an analytically fibered Hamiltonian with purely
absolutely continuous spectrum. Obtaining Mourre estimates with a
suitable conjugate operator, we prove, in particular, that the
singular continuous spectrum of ${\mathcal H}_\theta$ is empty.
\end{abstract}

\maketitle




\section{Introduction}

Let $\omega \subset \rd$ be a bounded domain  with boundary
$\partial \omega \in C^2$.  Denote by
$\Omega : = \omega \times \re$ the straight tube in $\rt$. For a given
$\theta\in C^1(\re,\re)$ we define the twisted tube $\Omega_\theta$ by
$$
\Omega_\theta = \left\{ r_\theta(x_3)\, x \in \re^3\,| \, x = (x_1,
x_2, x_3) \in \re^3,\  x_\omega : = (x_1, x_2) \in\omega\right\},
$$
where
$$
r_\theta(x_3) = \left( \begin{array}{rcc}
\cos \theta(x_3) & \sin \theta(x_3) & 0 \\
-\sin \theta(x_3) & \cos \theta(x_3) & 0 \\
0 & 0& 1 \end{array} \right).
$$
 We define the Dirichlet Laplacian
${\mathcal H}_\theta$ as the unique self-adjoint operator
generated in ${\rm L}^2(\Omega_\theta)$  by the closed quadratic form
\begin{equation}
{\mathcal Q}_{\theta}[u] : = \int_{\Omega_\theta} |\nabla u|^2\, d
\bx , \qquad u\in {\rm D}({\mathcal Q}_\theta) : = \h_0^1(\Omega_\theta).
\end{equation}
In fact, we do not work directly with ${\cal H}_\theta$, but rather with a unitarily equivalent operator $H_{\theta'}$ acting in the straight tube $\Omega$, see \eqref{gr1}. The related unitary transformation is generated by a change of variables which maps the twisted tube $\Omega_\theta$ onto the straight tube $\Omega$, see equation \eqref{gr21}.

The goal of the present article is to study the nature of the
essential spectrum of the operator ${\mathcal H}_\theta$ under
appropriate assumptions about the twisting angle $\theta$.
Although the spectral properties of a twisted waveguide have been intensively studied
in recent years, attention has been paid mostly to the
discrete spectrum of ${\cal H}_\theta$, \cite{BrKoRaSo, ExKo, Gr},
or to the Hardy inequality for ${\cal H}_\theta$, \cite{ekk}.

\noindent In this article we  discuss the influence of twisting on
the nature of the essential spectrum of ${\cal H}_\theta$. First,
we show that if the difference $\theta_1' - \theta_2'$ decays fast
enough as $|x_3| \to \infty$, then the wave operators for the
operator pair $(H_{\theta_1'}, H_{\theta_2'})$ exist and are
complete, and in particular, the absolutely continuous spectra of
$H_{\theta_1'}$ and  $H_{\theta_2'}$ coincide. Further, we observe
that if $\theta'=\beta$ is constant, then the operator $H_\beta$
is analytically fibered, cf. \eqref{hbeta}, and therefore its
singular continuous spectrum is empty, \cite{FiSo,GeNi}. Assuming
that $\theta'(x_3) = \beta -\eps(x_3)$ with $\eps\in
C^1(\re,\re)$, we then show that if $\eps$ decays fast enough at
infinity, then $H_{\theta'}$ has no singular continuous spectrum,
see Theorem \ref{main}. The proof of Theorem \ref{main} is based
on the Mourre commutator method, \cite{Mo,pss,ABG}. We construct
a suitable conjugate operator $A$ and show that the commutator
$[H_{\theta'}, iA]$ satisfies a Mourre estimate on  sufficiently small intervals
outside a discrete subset of $\re$, Theorem \ref{mourre-estim}.
The construction of the conjugate operator is based on a careful
analysis of the band functions $E_n(k)$ of the unperturbed
operator $H_\beta$, $k\in\re$ being the Fourier variable dual to
$x_3$. A similar strategy was used in \cite{GL, ABBFR, BrRaSo,KT},
where the generator of dilations in the longitudinal direction of
the waveguide was used as a conjugate operator. However, in the
situations studied  in these works the associated band functions
have a non zero derivative everywhere except for the origin. In
our model, contrary to \cite{GL, ABBFR, BrRaSo,KT}, the band
functions $E_n$ may have many stationary points. In addition, we
have to take into account possible crossing points between
different band functions. The generator of dilations therefore
cannot be used as a conjugate operator in our case, and a
different approach is needed. Our conjugate operator acts in the
fibered space as
\begin{equation} \label{conj}
\frac{i}{2}\left( \gamma(k)\, \pd_k + \pd_k\, \gamma(k)\right)
\end{equation}
where $\gamma\in C^\infty_0(\re;\re)$ is a suitably chosen
function, whose particular form depends on the interval on which
the Mourre estimate is established, see Theorem \ref{thm1}.

\noindent We would like to mention that a general theory of Mourre
estimates for analytically fibered operators  and their appropriate perturbations was developed in
\cite{GeNi}.
The situation with  the twisted waveguide analyzed  in the present article  is much more specific than the general abstract scheme
studied in \cite{GeNi}. Hence, although the
construction in \eqref{conj} is influenced in some extent by \cite{GeNi}, our conjugate operator is essentially different from the one used in \cite{GeNi}, and is considerably more useful for our purposes. In particular, the construction of this quite explicit  conjugate operator allows us to handle the specific second-order differential perturbation which arises in the context of the twisted waveguide, and to  verify all the regularity conditions for
$e^{itA}, [H_{\theta'}, iA]$ and $[[H_{\theta'}, iA], iA]$ needed
for the passage from the Mourre estimate to the absence of the
singular continuous spectrum, see Proposition \ref{conditions}.
We have thus been able to apply the Mourre theory to the perturbed operator $H_{\theta'}$,
and to find simple and efficient sufficient conditions on
$\varepsilon$ under which the singular continuous spectrum of $H_{\theta'}$ is
empty. We therefore believe that our construction of the conjugate
operator might be of independent interest.

\noindent The article is organized as follows. In Section
\ref{sect-main} we state our main results. In Section
\ref{sect-domain} we prove Proposition \ref{op-domain} describing
the domain of the operator ${\mathcal H}_\theta$. In Section
\ref{sect-wave} we prove Theorem \ref{t1} which entails the
existence and the completeness of the wave operators for the
operator pair $(H_{\theta_1'},H_{\theta_2'})$ for appropriate
$\theta_1' - \theta_2'$, and hence the coincidence of $\sigma_{\rm
ac}(H_{\theta_1'})$ and $\sigma_{\rm ac}(H_{\theta_2'})$. In
Section \ref{sect-const-twist} we assume that the twisting is
constant, i.e. $\theta' = \beta$ and examine the spectral and
analytical properties of the fiber family $h_\beta(k)$, $k  \in
\re$. In Section \ref{sect-conjugate} we construct the conjugate
operator needed for the subsequent Mourre estimates. In Section
\ref{sect-mourre} we obtain Mourre estimates for the case of a
constant twisting. Finally, in Section \ref{sect-perturbation} we
extend these estimates to the case of $\theta' = \beta -
\varepsilon$ where $\beta \in \re$, and $\varepsilon$ decays fast
enough together with its first derivative.

\section{Main results}
\label{sect-main}
\subsection{Notation} \label{subsect-notation}
\noindent Let us fix some notation. Given a measure space $(M,
{\mathcal A},\mu)$, we denote by $\id_M$  the identity operator in
${\rm L}^2(M) = {\rm L}^2(M;d\mu)$. Further,  we will denote by
$(u,v)_{{\rm L}^2(M)} = \int_M \bar{u}\, v d\mu$ the scalar
product in ${\rm L}^2(M)$ and by $\|u\|_{{\rm L}^p(M)}$, $p \in
[1,\infty]$, the ${\rm L}^p$-norm of $u$.  If there is no risk of
confusion we will drop the indication to the set $M$ and write
$(u,v)$ and $\|u\|_p$ instead in order to simplify the notation.
Given a set $M$ and two functions $f_1,\, f_2:M\to\re$, we write
$f_1(m) \asymp f_2(m), \ m\in M$, if there exists a constant  $c \in (0,\infty)$ such that for each $m \in M$ we have
$$
c^{-1}\, f_1(m) \, \leq \, f_2(m) \, \leq \, c\, f_1(m).
$$
\noindent Given a separable Hilbert space  $X$, we denote by
$\mathcal{L}(X)$ (resp., $S_{\infty}(X))$ the class of bounded
(resp., compact) linear operators acting in $X$. Similarly, by
$S_p(X)$, $p \in [1,\infty)$, we denote the Schatten-von Neumann
classes of compact operators acting in $X$; we recall that the
norm in $S_p(X)$ is defined as $\|T\|_{S_p} : = \left({\rm
Tr}\,(T^* T)^{p/2}\right)^{1/p}$, $T \in S_p(X)$. In particular,
$S_1$ is the trace class, and $S_2$ is the Hilbert-Schmidt class.
Moreover, if $T$ is a self-adjoint operator acting in $X$, we
denote by ${\rm D}(T)$ the operator domain of $T$. Finally, for
$\alpha \in \re$ define the function
    \bel{ngr1}
\phi_\alpha(s) : = (1 + s^2)^{-\alpha/2}, \quad s \in \re.
    \ee

\subsection{Domain issues.} \label{subsect-domain} Our first result shows
that if both $\theta'$ and $\theta''$ are
continuous and bounded, then the domain of the operator ${\mathcal
H}_\theta$ coincides with $\h^2(\Omega_{\theta}) \cap
\h_0^1(\Omega_{\theta})$.

\begin{prop} \label{op-domain} Assume that $\omega \subset \rd$ is a bounded domain with
    boundary $\partial \omega \in C^2$, and $\theta \in C^2(\re)$
    with $\theta', \theta'' \in {\rm L}^\infty(\re)$. Then
    \bel{gr18}
    {\rm D}\,({\mathcal H}_\theta) = \h^2(\Omega_\theta) \cap
    \h_0^1(\Omega_\theta).
    \ee
\end{prop}
\noindent Proposition \ref{op-domain} could be considered as a
fairly standard result but since we have not been able to find in the
literature a version suitable for our purposes (most of the
references available treat bounded domains or the complements of
compact sets), we include a detailed sketch of the proof in
Section \ref{sect-domain}.

\noindent Next, we introduce the  operator $U_\theta : {\rm
L}^2(\Omega_\theta) \to
 {\rm L}^2(\Omega)$ generated by the change of variables
    \bel{gr21}
    \Omega
    \ni x \mapsto r_\theta(x_3)\, x \in \Omega_\theta.
    \ee
Namely, for $w \in {\rm L}^2(\Omega_\theta)$ set
$$
(U_\theta \, w)(x) = w \left( r_\theta(x_3)\, x \right), \quad x
\in \Omega.
$$
Evidently, $U_\theta : {\rm L}^2( \Omega_\theta) \to {\rm
L}^2(\Omega)$ is unitary since \eqref{gr21} defines a
diffeomorphism  whose Jacobian is identically equal to one. Now
assume $g \in C(\re; \re) \cap {\rm L}^\infty(\re)$ and introduce
the quadratic form
$$
Q_g[u] = \int_{\Omega} \left(|\nabla_\omega u|^2+|\pd_3 u+ g\,
\pd_\tau u|^2 \right){\rm d}x, \quad u\in {\rm D}(Q_g) =
\h_0^1(\Omega),
$$
where  $\nabla_\omega : = (\pd_1, \pd_2)^T$, and $\pd_\tau : =
x_1\pd_2-x_2\pd_1$. Denote by $H_{g}$  the self-adjoint operator
generated in ${\rm L}^2(\Omega)$ by the closed quadratic form
$Q_g$. The transformation $U_\theta$ also maps
$\h_0^1(\Omega_\theta)$ bijectively onto $\h_0^1(\Omega)$.
Hence, for $g=\theta'$ we get
$$
\mathcal{Q}[w] = Q_{\theta'}[U_\theta\, w], \quad w \in
\h_0^1(\Omega_\theta),
$$
which implies
\bel{gr1}
H_{\theta'} = U_\theta\,  {\mathcal H}_\theta\,
    U_\theta^{-1}.
\ee

\noindent Assume now that $g \in C^1(\re)$ with $g, g' \in {\rm
L}^\infty(\re)$. Set $G(x_3) : = \int_0^{x_3} g(s) ds$, $x_3 \in
\re$. Then $U_G$ maps
     bijectively $\h^2(\Omega_G)$ onto $\h^2(\Omega)$. Therefore, Proposition \ref{op-domain} and the unitarity of
     $U_G: {\rm L}^2(\Omega_G) \to {\rm L}^2(\Omega)$ imply the following
     \begin{follow} \label{fgr1}
     Assume that $\omega \subset \rd$ is a bounded domain with
    boundary $\partial \omega \in C^2$, and $g \in  C^1(\re)$
    with $g, g' \in {\rm L}^\infty(\re)$. Then the domain of the operator
    $H_g$ coincides with $\h^2(\Omega) \cap
    \h_0^1(\Omega)$.
    \end{follow}
   \noindent  Furthermore, if $g \in C^1(\re)$ with $g, g' \in {\rm L}^\infty(\re)$ we have
    \bel{gr26}
    H_g\,  u = \left(-\pd_1^2 - \pd_2^2 - (\pd_3 + g\, \pd_\tau)^2\right)u , \quad u \in \h^2(\Omega) \cap
    \h_0^1(\Omega),
    \ee
    since ${\mathcal H}_G\,  \varphi = - \Delta \varphi$, $\varphi \in \h^2(\Omega_G) \cap
    \h_0^1(\Omega_G)$.

\subsection{Existence and completeness of the wave operators.} \label{subsect-wave}
Next we show that under appropriate assumptions on the
difference $g_1 - g_2$, the wave operators for the operator pair
$(H_{g_1}, H_{g_2})$ exist and are complete, and hence the
absolutely continuous spectra of the operators $H_{g_1}$ and
$H_{g_2}$ coincide.
\begin{theorem} \label{t1}
Assume that $\omega \subset \rd$ is a bounded domain with
$C^2$-boundary. Let $g_j \in C^1(\re ; \re)$ with $g_j, g'_j \in
{\rm L}^\infty(\re)$, $j =1,2$. Suppose that there exists $\alpha
> 1$  such that
    \bel{gr31}
    \|\phi_{-\alpha}(g_1 - g_2)\|_{{\rm L}^{\infty}(\re)} < \infty,
    \ee
the function $\phi_\alpha$ being defined in \eqref{ngr1}. Then we have
    \bel{gr30}
 H_{g_1}^{-2} - H_{g_2}^{-2}  \in S_1({\rm L}^2(\Omega)).
    \ee
\end{theorem}

\noindent Theorem \ref{t1} is proven in Section
\ref{sect-wave}. By a classical result from the stationary
scattering theory (see the original work \cite{Bi} or
\cite[Corollary 3, Section 3, Chapter XI]{RSIII}, \cite[Chapter 6, Section 2, Theorem 6]{y}), this theorem
implies the following

\begin{follow} \label{fgr2}
Under the assumptions of Theorem \ref{t1} the wave operators
$$
{\rm s}-\lim_{t \to \pm \infty} e^{itH_{g_1}} e^{-itH_{g_2}} P_{\rm ac}(H_{g_2})
$$
for
the operator pair $(H_{g_1}, H_{g_2})$ exist and are complete.
Therefore, the absolutely continuous parts of   $H_{g_1}$ and
$H_{g_2}$ are unitarily equivalent, and, in particular,
    \bel{gr32}
    \sigma_{\rm ac}(H_{g_1}) = \sigma_{\rm ac}(H_{g_2}).
    \ee
\end{follow}
\noindent Corollary \ref{fgr2} admits an equivalent formulation in terms of the operator pair $({\mathcal H}_{\theta_1}, {\mathcal H}_{\theta_2})$:

\begin{follow} \label{fgr3}
Assume that $\omega \subset \rd$ is a bounded domain with
$C^2$-boundary. Let $\theta_j \in C^2(\re ; \re)$ with $\theta_j, \theta'_j, \theta''_j \in
{\rm L}^\infty(\re)$, $j =1,2$. Suppose that there exists $\alpha
> 1$  such that
    $$
    \|\phi_{-\alpha}(\theta'_1 - \theta'_2)\|_{{\rm L}^{\infty}(\re)} < \infty,
    $$
Then the wave operators
$$
{\rm s}-\lim_{t \to \pm \infty} e^{it{\mathcal H}_{\theta_1}} {\mathcal J} e^{-it{\mathcal H}_{\theta_2}} P_{\rm ac}({\mathcal H}_{\theta_2}), \quad {\mathcal J} : = U_{\theta_1}^{-1} U_{\theta_2},
$$
for
the operator pair $({\mathcal H}_{\theta_1}, {\mathcal H}_{\theta_2})$ exist and are complete.
Therefore, the absolutely continuous parts of   ${\mathcal H}_{\theta_1}$ and
${\mathcal H}_{\theta_2}$ are unitarily equivalent, and, in particular,
    $\sigma_{\rm ac}({\mathcal H}_{\theta_1}) = \sigma_{\rm ac}({\mathcal H}_{\theta_2})$.
    \end{follow}

\subsection{Constant twisting} \label{subsect_const-twist}

\noindent In our remaining results, we concentrate on the case of
appropriate perturbations of a constant twisting, i.e. the case where $\theta'$ is equal to a constant $\beta\in\re$.
First, we consider the unperturbed operator $H_\beta$.  We define the
partial Fourier transform $\F$, unitary in ${\rm L}^2(\Omega)$, by
$$
(\F\, u)(x_\omega,k) = (2\pi)^{-1/2}\, \int_{\re} e^{-i k x_3}\,
u(x_\omega,x_3)\, dx_3, \qquad k\in\re, \quad x_\omega \in \omega.
$$
Then we have
\begin{equation} \label{hbeta}
\hat H_\beta = \F\, H_\beta\, \F^*= \int _\re ^\oplus  h_\beta(k)\, dk,
\end{equation}
where, by \eqref{gr26} with $g = \beta$, the operator $h_\beta(k)$
 acts on its domain
${\rm D}\,(h_\beta(k))= \h^2(\omega)\cap\h^1_0(\omega)$ as
$$
h_\beta(k) = - \Delta_\omega  + ( \beta i\partial_{\tau} - k)^2,
$$
$-\Delta_\omega$ being the self-adjoint operator generated in
${\rm L}^2(\omega)$ by the closed quadratic form
$$
\int_\omega |\nabla v|^2 dx_\omega,  \quad v \in \h_0^1(\omega).
$$
 Note that for all $k \in \re$
the resolvent $h_\beta(k)^{-1}$ is  compact,  and $h_\beta(k)$ has
a purely discrete spectrum. Let
\begin{equation} \label{En}
0< E_1(k) \leq  E_2(k) \leq \dots \leq E_n(k) \leq \dots , \qquad k\in\re,
\end{equation}
be the non-decreasing sequence of the eigenvalues of $h_\beta(k)$.
Denote by $p_n(k)$ the orthogonal projection onto ${\rm
Ker}(h_\beta(k) - E_n(k))$, $k \in \re$ and $n \in {\mathbb N}$.
By \cite{BrKoRaSo,ExKo} we have
\begin{equation} \label{spectra}
\sigma(H_\beta) =\sigma_{\rm ac}(H_\beta) = [E_1(0), \infty).
\end{equation}
A detailed discussion of the properties of $E_n(k)$ is given in
Section \ref{sect-const-twist}. It turns out that the functions
$E_n(k)$ are piecewise analytic, and that for any given $k_0\in
\re$, the function $E_n(k)$ can be analytically extended into an
open neighborhood of $k_0$. We denote such an extension by $\tilde
E_{n,k_0}(k)$. If $k_0$ is a point where $E_n(k)$ is analytic,
then of course $\tilde E_{n,k_0}(\cdot)=E_n(\cdot)$. With this
notation at hand, we introduce the following subsets of $\re$:
\begin{align*}
\E_1  & := \big\{ E \in \re  \, : \exists\, n\in\N, \, \exists\,  k_0 \in \re\,
 :  E_n(k_0) =E   \ \wedge \ \partial_k \tilde E_{n,k_0}(k_0) =0 \big\},
\\
\E_2 & := \big\{ E \in \re\, :  \exists\, k_0 \in \re ,\,
\exists\,  n, m \in \N, \, n\neq m \, :  E_n(k_0) = E_m(k_0)  =E \ \wedge\ \\
& \ \qquad \qquad \qquad \wedge\  \partial_k\tilde E_{n,k_0}(k_0)\, \partial_k\tilde E_{m,k_0}(k_0) < 0 \big\}.
\end{align*}
We  then define the set $\E$ of critical levels as follows:
\begin{equation} \label{c-levels}
\E=  \E_1 \cup \E_2.
\end{equation}

\begin{lemma} \label{loc-finite}
The set $\E$ is locally finite.
\end{lemma}

\noindent The proof of Lemma \ref{loc-finite} is given in Section
\ref{sect-const-twist}, immediately after Lemma \ref{l4}.

\subsection{Absence of singular continuous spectrum of $H_{\beta -
\varepsilon}$} \label{subsect-no-sing}

\begin{theorem} \label{main}
Let $\theta'(x_3) =\beta -\eps(x_3)$, where $\eps\in C^1(\re,\re)$  is such that
\begin{equation} \label{decay-eps}
\|\eps\, \phi_{-2}\|_\infty + \|\eps'\, \phi_{-2} \|_\infty <
\infty,
\end{equation}
the function $\phi_\alpha$ being defined in \eqref{ngr1}. Then:
\begin{itemize}
\item[(a)] Any compact subinterval of $\re\setminus\E$ contains at
most finitely many eigenvalues of $H_{\theta'}$, each having
finite multiplicity;

\smallskip

\item[(b)] The point spectrum of $H_{\theta'}$ has no accumulation points in $\re\setminus\E$;

\smallskip

\item[(c)] The singular continuous spectrum of $H_{\theta'}$ is empty.
\end{itemize}
\end{theorem}

\smallskip

\noindent Theorem \ref{main} is proven in Subsection \ref{subsect-sc}.

\begin{remark}
If $\eps\in C^1(\re,\re)$ is such that $\eps'$ is bounded and
$\|\eps\, \phi_{-\alpha} \| _\infty < \infty$ for some $\alpha>1$,
then  Corollary \ref{fgr2} and equation \eqref{spectra} imply
$$
\sigma_{\rm ac}(H_{\theta'}) =\sigma_{\rm ac}(H_{\beta})= [E_1(0), \infty).
$$
Note that in order to prove the absence of singular continuous
spectrum of $H_{\theta'}$ we need stronger hypothesis on $\eps$
and $\eps'$, see equation \eqref{decay-eps}.
\end{remark}

\noindent By \cite[Section XI.3]{RSIII}, Corollary \ref{fgr2} and Theorem \ref{main} part (c) imply

\begin{follow} \label{as-complete}
Under the assumptions of Theorem \ref{main} the wave operators for the
operator pair $(H_\beta, H_{\theta'})$ exist and are
asymptotically complete.
\end{follow}

\section{Proof of Proposition \ref{op-domain}} \label{sect-domain}
\noindent Denote by $C_0^{\infty}(\overline{\Omega_\theta})$ the
class of functions $u \in C^{\infty}(\overline{\Omega_\theta})$,
compactly supported in $\overline{\Omega_\theta}$. Set
$$
\dcot : = \left\{u \in C_0^{\infty}(\overline{\Omega_\theta}) \, |
\, u_{|\partial \Omega_\theta} = 0\right\}.
$$
\begin{lemma} \label{lgr1}
  Under the assumptions of Proposition \ref{op-domain} there
exists a constant $c \in (0,\infty)$ such that
    \bel{gr2}
    \|u\|_{\h^2(\Omega_\theta)}^2 \leq c\int_{\Omega_\theta}
    (|\Delta u|^2 + |u|^2) dx
    \ee
    for any $u \in \dcot$.
    \end{lemma}
    \begin{proof}
    Our argument will follow closely the proof of  \cite[Chapter 3, Lemma 8.1]{lu}.
    We have
    \bel{gr7}
    \int_{\Omega_\theta}
    (|\Delta u|^2 + c_0 |u|^2) dx = \int_{\Omega_\theta}\left(\sum_{j,k
    = 1}^3 |\partial_j \partial_k u|^2 + c_0 |u|^2\right) dx +  2\int_{\partial
    \Omega_\theta}K \left|\frac{\partial u}{\partial \nu}\right|^2
    dS, \quad u \in \dcot,
    \ee
     (see \cite{lu} or \cite[Chapter 5, Section 5, Problem 6]{tay}) where $c_0 \in (0,\infty)$ is an arbitrary constant which is
    to be specified later, $K$ is the mean curvature, and $\nu$ is
    the exterior normal unit vector at $\partial
    \Omega_\theta$. Our assumptions on $\partial \omega$ and
    $\theta$ imply  that  for any $u \in
    \dcot$  we have
    \bel{gr5}
     2\int_{\partial
    \Omega_\theta}K \left|\frac{\partial u}{\partial \nu}\right|^2
    dS \geq - c_1 \int_{\partial
    \Omega_\theta} |\nabla u|^2 dS.
    \ee
    with
    $$
    c_1 : = 2 \sup_{x \in \partial \Omega_\theta} |K(x)| \leq 2 \sup_{(x_\omega, x_3) \in \partial \omega \times \re} \left\{(|\theta''(x_3)|
    + \theta'(x_3)^2) |x_\omega| + (1 + \theta'(x_3)^2 |x_\omega|^2) |\kappa(x_\omega)|\right\},
    $$
    where $\kappa(x_\omega)$ is the curvature of $\partial \omega$ at the point $x_\omega \in \partial \omega$.
    Let us check that for any $\varepsilon > 0$ there
    exists a constant $c_2(\varepsilon)$ such that for any $v \in
    C^{\infty}_0(\overline{\Omega_\theta})$ we have
    \bel{gr2a}
    \int_{\partial \Omega_\theta} |v|^2 dS \leq \int_{
    \Omega_\theta} \left(\varepsilon |\nabla_\omega v|^2 + c_2(\varepsilon) |v|^2\right)
    dx
    \ee
    where, as above, $\nabla_\omega : = (\partial_1, \partial_2)^T$.
    In order to prove this, we note the  inequality
    \bel{gr4}
    \int_{\partial \Omega_\theta} |v|^2 dS \leq c_3 \int_\re
    \left(\int_{\partial \omega_{\theta(x_3)}} |v|^2 ds\right) dx_3
    \ee
    where
     $$
    c_3 : = \sup_{(x_\omega, x_3) \in \partial \omega \times \re} \left(1 +
    \theta'(x_3)^2 |x_\omega|^2\right)^{1/2},
    $$
    and $\omega_{\theta(a)}$ is the cross-section of
    $\Omega_\theta$ with the plane $\{x_3 = a\}$, $a \in \re$. \\
    Next, since $\omega$ is a bounded domain with sufficiently regular boundary, we find
    that for any $\varepsilon > 0$ there exists a constant
    $c_4(\varepsilon)$ such that for any $x_3 \in \re$ and any $w
    \in C^{\infty}(\overline{\omega_{\theta(x_3)}})$ we have
    \bel{gr3}
    \int_{\partial \omega_{\theta(x_3)}} |w|^2 ds \leq
    \int_{\omega_{\theta(x_3)}} \left(\varepsilon |\nabla w|^2 + c_4(\varepsilon)|w|^2\right)
    dx_\omega
    \ee
   (see e.g.  \cite[Chapter 2, Eq. (2.25)]{lu}).  Choosing $w = v(\cdot, x_3)$ in \eqref{gr3},
   integrating with respect to $x_3$, and bearing in mind
    \eqref{gr4}, we get
    $$
    \int_{\partial \Omega_\theta} |v|^2 dS \leq \int_{
    \Omega_\theta}\left(c_3 \varepsilon |\nabla_\omega v|^2 + c_3
    c_4(\varepsilon)|v|^2\right) dx
    $$
     which implies \eqref{gr2a} with $c_2(\varepsilon) = c_3
    c_4(\varepsilon /c_3)$. Now the combination of \eqref{gr5}
    and \eqref{gr2a} implies
    \bel{gr8}
     2\int_{\partial
    \Omega_\theta}K \left|\frac{\partial u}{\partial \nu}\right|^2
    dS \geq - c_1 \int_{\Omega_\theta} \left(\varepsilon
    \sum_{j,k=1}^3|\partial_j \partial_k u|^2  + c_2(\epsilon)
    |\nabla u|^2\right) dx.
    \ee
    Further, we have
    \bel{gr9}
    \int_{\Omega_\theta} |\nabla u|^2 dx = - {\rm Re}\,\int_{\Omega_\theta} \Delta u \overline{u}
    dx \leq \frac{1}{2}\int_{\Omega_\theta}\left(|\Delta u|^2 +
    |u|^2\right) dx.
    \ee
    Combining \eqref{gr7}, \eqref{gr8}, and \eqref{gr9}, we find
    that for any $\varepsilon > 0$ we have
    $$
    \int_{\Omega_\theta}\left(\left(1+c_1 c_2(\varepsilon)/2\right) |\Delta u|^2 +
    c_0 |u|^2\right) dx \geq
    $$
    $$
    \int_{\Omega_\theta}\left(\left(1- c_1 \varepsilon\right)
    \sum_{j,k=1}^3|\partial_j \partial_k u|^2 +
    \left(c_0 - c_1 c_2(\varepsilon)/2\right) |u|^2\right)dx
    $$
    which yields \eqref{gr2} under appropriate choice of $c_0, c$
    and $\varepsilon$.
    \end{proof}
   \noindent  Denote by $\tilde{\rm H}^2(\Omega_\theta)$ the Hilbert space $\left\{u \in \h_0^1(\Omega_\theta) \, | \, \Delta u \in {\rm L}^2(\Omega_\theta)\right\}$
    with  scalar product generated by the quadratic form $\int_{\Omega_\theta}
    (|\Delta u|^2 + |u|^2) dx$. \\

    \begin{lemma} \label{lgr2}
    Under the assumptions of Proposition \ref{op-domain} we have $u \in \tilde{\rm H}^2(\Omega_\theta)$
    if and only if $u \in \h^2(\Omega_\theta) \cap \h_0^1(\Omega_\theta)$.
    \end{lemma}
    \begin{proof}
    By
    \bel{aug1}
    \int_{\Omega_\theta}
    (|\Delta u|^2 + |u|^2) dx \leq 3 \|u\|^2_{\h^2(\Omega_\theta)}, \quad u \in \h^2(\Omega_\theta),
    \ee
    and \eqref{gr2}, we have
    \bel{gr25}
    \int_{\Omega_\theta}(|\Delta u|^2 + |u|^2) dx \asymp \|u\|^2_{\h^2(\Omega_\theta)}, \quad u \in \dcot.
    \ee
    Evidently, the class $\dcot$ is dense in $\h^2(\Omega_\theta) \cap
    \h_0^1(\Omega_\theta)$. Then \eqref{aug1} easily implies that $\dcot$
    is  dense in $\tilde{\rm H}^2(\Omega_\theta)$ as well. Now the claim of the lemma follows from
    \eqref{gr25}.
    \end{proof}
    \begin{proof}[Proof of Proposition \ref{op-domain}]
    Let $L$ be the operator $-\Delta$ with domain
    $C_0^\infty(\Omega_\theta)$, and $L^*$ be the adjoint of $L$ in
    ${\rm L}^2(\Omega_\theta)$. If $v \in {\rm D}(L^*)$, then a
    standard  argument from the theory of distributions
    over  $C_0^\infty(\Omega_\theta)$ shows that $L^*v = -\Delta v
    \in {\rm L}^2(\Omega_\theta)$. Since ${\mathcal H}_\theta$ is a
    restriction of $L^*$, we find that $u \in {\rm D}\,({\mathcal
    H}_\theta)$ implies that ${\mathcal H}_\theta u = -\Delta u
    \in {\rm L}^2(\Omega_\theta)$. On the other hand, $u \in {\rm D}\,({\mathcal
    H}_\theta)$ implies $u \in {\rm D}\,({\mathcal
    Q}_\theta) = \h_0^1(\Omega_\theta)$. By Lemma \ref{lgr2} we
    have $u \in \h^2(\Omega_\theta) \cap
    \h_0^1(\Omega_\theta)$, i.e.
    \bel{gr19}
    {\rm D}\,({\mathcal H}_\theta) \subseteq \h^2(\Omega_\theta) \cap
    \h_0^1(\Omega_\theta).
    \ee
    If we now suppose that
    \bel{gr20}
    {\rm D}\,({\mathcal H}_\theta) \neq \h^2(\Omega_\theta) \cap
    \h_0^1(\Omega_\theta),
    \ee
    then \eqref{gr19} and \eqref{gr20} would imply that the operator ${\mathcal H}_\theta$ has a
    proper symmetric extension, namely the operator $-\Delta$ with
    domain $\h^2(\Omega_\theta) \cap
    \h_0^1(\Omega_\theta)$, which contradicts the
    self-adjointness of ${\mathcal H}_\theta$. Therefore,
    \eqref{gr18} holds true, and the proof of Proposition \ref{op-domain} is complete.  \end{proof}

\section{Proof of Theorem \ref{t1}}\label{sect-wave}
\noindent For the proof of Theorem \ref{t1} we need an auxiliary
result, Lemma \ref{lgr3}, preceded by some necessary notation.

\noindent  Let $\{\mu_j\}_{j \in {\mathbb N}}$ be the the
non-decreasing sequence of the eigenvalues of the operator
$-\Delta_\omega$. Since $H_g \geq \mu_1 \id_{\Omega}$, and $\mu_1
>0$, the operator $H_g$ is invertible.\\

    \begin{lemma} \label{lgr3}
Let $g \in C^1(\re ; \re)$ with $g, g' \in {\rm L}^\infty(\re)$.\\
(i) Assume $f \in {\rm L}^2(\re)$. Then we have
    \bel{gr33}
f(x_3) H_g^{-1} \in S_2({\rm L}^2(\Omega)).
    \ee
    (ii) Assume $h \in {\rm L}^4(\re)$. Then we have
    \bel{gr34}
h(x_3) \partial_j H_g^{-1} \in S_4({\rm L}^2(\Omega)), \quad j=1,2,3.
    \ee
\end{lemma}
\begin{proof}
By Corollary \ref{fgr1} the operator $H_0 H_g^{-1}$ is bounded, so
that it suffices to prove \eqref{gr33} -- \eqref{gr34} for $g=0$.
Evidently,
\begin{align} \label{gr35}
& \|f H_0^{-1}\|^2_{S_2({\rm L}^2(\Omega))} = \sum_{j \in {\mathbb N}} \|f (-\pd_3^2 + \mu_j)^{-1}\|^2_{S_2({\rm L}^2(\re))} = \nonumber \\
& =(2\pi)^{-1} \sum_{j \in {\mathbb N}}\int_\re |f(s)|^2
ds \int_\re \frac{d\xi}{(\xi^2 + \mu_j)^2} = (2\pi)^{-1} \sum_{j
\in {\mathbb N}}\mu_j^{-3/2} \int_\re |f(s)|^2 ds \int_\re
\frac{d\xi}{(\xi^2 + 1)^2}.
\end{align}
Set ${\mathcal N}(\lambda) : =
\#\{j\in {\mathbb N}\, | \, \mu_j < \lambda\}$, $\lambda > 0$. By
the celebrated Weyl law, we have ${\mathcal N}(\lambda) =
\frac{|\omega|}{4\pi} \lambda(1 + o(1))$ as $\lambda \to \infty$
where $|\omega|$ is the area of $\omega$ (see the original work
\cite{w} or \cite[Theorem XIII.78]{RSIV}). Therefore, the series
$\sum_{j \in {\mathbb N}}\mu_j^{-\gamma} = \gamma \int_{\mu_1}^{\infty}
\lambda^{-\gamma - 1} {\mathcal N}(\lambda) d\lambda$ is
convergent if and only if $\gamma > 1$. In particular,
    \bel{gr36}
    \sum_{j \in {\mathbb N}}\mu_j^{-3/2} < \infty,
    \ee
    so that the r.h.s. of \eqref{gr35} is finite which implies \eqref{gr33} with $g=0$.\\
    Let us now prove \eqref{gr34} with $g=0$ and $j=1,2$.  We have
    $$
    h \partial_j H_0^{-1} =
    \pd_j ((-\Delta_\omega) \otimes \id_\re)^{-1/2}h((-\Delta_\omega) \otimes \id_\re)^{1/2}H_0^{-1}.
    $$
    Since the operators $\pd_j ((-\Delta_\omega) \otimes \id_\re)^{-1/2}$, $j=1,2$, are bounded,
    it suffices to show that
    \bel{gr37}
    h((-\Delta_\omega) \otimes \id_\re)^{1/2}H_0^{-1} \in S_4({\rm L}^2(\Omega)).
    \ee
    Applying a standard interpolation result (see e.g. \cite[Theorem 4.1]{sim} or \cite[Section 4.4]{BeLo}), and bearing in mind \eqref{gr36}, we get
    $$
\|h((-\Delta_\omega) \otimes \id_\re)^{1/2}H_0^{-1}\|^4_{S_4({\rm L}^2(\Omega))} = \sum_{j \in {\mathbb N}} \mu_j^2 \|h (-\pd_3^2 + \mu_j)^{-1}\|^4_{S_4({\rm L}^2(\re))} \leq
$$
$$
\leq (2\pi)^{-1} \sum_{j \in {\mathbb N}} \mu_j^2 \int_\re |h(s)|^4 ds \int_\re \frac{d\xi}{(\xi^2 + \mu_j)^4} =
(2\pi)^{-1} \sum_{j \in {\mathbb N}}\mu_j^{-3/2} \int_\re |h(s)|^4 ds \int_\re \frac{d\xi}{(\xi^2 + 1)^4} < \infty
$$
    which implies \eqref{gr37}. Finally, we prove \eqref{gr34} with $g=0$ and $j=3$. To this end it suffices to apply
     again \cite[Theorem 4.1]{sim} and \eqref{gr36}, and get
    $$
\|h \pd_3 H_0^{-1}\|^4_{S_4({\rm L}^2(\Omega))} = \sum_{j \in {\mathbb N}}  \|h \pd_3(-\pd_3^2 + \mu_j)^{-1}\|^4_{S_4({\rm L}^2(\re))} \leq
$$
$$
\leq (2\pi)^{-1} \sum_{j \in {\mathbb N}}  \int_\re |h(s)|^4 ds \int_\re \frac{\xi^4\,d\xi}{(\xi^2 + \mu_j)^4}  =
(2\pi)^{-1} \sum_{j \in {\mathbb N}}\mu_j^{-3/2} \int_\re |h(s)|^4 ds \int_\re \frac{\xi^4\, d\xi}{(\xi^2 + 1)^4} < \infty.
$$
\end{proof}

\begin{proof}[Proof of Theorem \ref{t1}]
For $z \in \rho(H_{g_1}) \cap  \rho(H_{g_2})$  we have
$$
(H_{g_1}- z)^{-2} - (H_{g_2}- z)^{-2}=  \frac
{\partial}{\partial z} (H_{g_1}- z)^{-1}
W(H_{g_2}- z)^{-1}=
$$
\begin{equation} \label{gr39}
(H_{g_1}- z)^{-2}W (H_{g_2}- z)^{-1}+
(H_{g_1}- z)^{-1}W (H_{g_2}-  z)^{-2}
\end{equation}
with
$$
W : = \pd_\tau(g_1^2-g_2^2)\pd_\tau +   \pd_3 (g_1 - g_2)\pd_\tau + \pd_\tau (g_1 - g_2)\pd_3.
$$
Choosing $z = 0$, we obtain
$$
 H_{g_1}^{-2} - H_{g_2}^{-2} =
 $$
 $$
 -\left(\phi_{3\alpha/4} \partial_\tau H_{g_1}^{-2}\right)^* \left((g_1 - g_2)
 \phi_{-\alpha} \phi_{\alpha/4} \partial_3 H_{g_2}^{-1} +
 (g_1^2 - g^2_2)\phi_{-\alpha} \phi_{\alpha/4} \partial_\tau H_{g_2}^{-1}\right)
 -
 $$
 $$
 \left(\phi_{3\alpha/4} \partial_3 H_{g_1}^{-2}\right)^*
 (g_1 - g_2)\phi_{-\alpha} \phi_{\alpha/4} \partial_\tau H_{g_2}^{-1}
 -
$$
$$
 \left(\phi_{\alpha/4} \partial_\tau H_{g_1}^{-1}\right)^*
 \left((g_1 - g_2) \phi_{-\alpha} \phi_{3\alpha/4} \partial_3 H_{g_2}^{-2} +
 (g_1^2 - g^2_2)\phi_{-\alpha} \phi_{3\alpha/4} \partial_\tau H_{g_2}^{-2}\right)
 -
 $$
 $$
 \left(\phi_{\alpha/4} \partial_3 H_{g_1}^{-1}\right)^*
 (g_1 - g_2) \phi_{-\alpha} \phi_{3\alpha/4} \partial_\tau H_{g_2}^{-2}.
 $$
 Since the multipliers by $(g_1 - g_2)\phi_{-\alpha}$ and $(g_1^2 - g^2_2)\phi_{-\alpha}$ are bounded operators by \eqref{gr31} and $g_j \in {\rm L}^\infty(\re)$, $j=1,2$, while
 $$
 \phi_{\alpha/4} \partial_\ell H_{g_j}^{-1} \in S_4({\rm L}^2(\Omega)), \quad \ell = \tau, 3, \quad j=1,2,
 $$
  by Lemma \ref{lgr3} (ii), it suffices to show that
  \bel{gr40}
 \phi_{3\alpha/4} \partial_\ell H_{g_j}^{-2} \in S_{4/3}({\rm L}^2(\Omega)), \quad \ell = \tau, 3,  \quad j=1,2.
 \ee
 In what follows we write $g$ instead of $g_j$, $j=1,2$. Commuting   multipliers by functions $\phi$ which depend only on $x_3$ and belong to appropriate H\"ormander classes, with the resolvent $H_g^{-1}$, and bearing in mind that
 $$
 [\phi, H_g^{-1}] = -H_g^{-1}(\phi'' + 2\phi'(\pd_3 + g \pd_\tau))H_g^{-1},
 $$
 we obtain
$$
\phi_{3\alpha/4} \partial_\tau H_{g}^{-2}
= \phi_{\alpha/4} \partial_\tau H_g^{-1} \phi_{\alpha/2} H_g^{-1} -
\phi_{\alpha/4} \partial_\tau H_g^{-1} \phi_{\alpha/2}'' H_g^{-2} + 2 \phi_{\alpha/4} \partial_\tau H_g^{-1} (\phi'_{\alpha/2})^2 \phi_{-\alpha/2} H_g^{-2}-
$$
\bel{gr41}
2 \phi_{\alpha/4} \partial_\tau H_g^{-1}\phi'_{\alpha/2} \phi_{-\alpha/2} \left(\partial_3 + g \partial_\tau\right) H_g^{-1} \left(\phi_{\alpha/2} H_g^{-1} - 2
\left(\phi_{\alpha/2}'' + {\phi'_{\alpha/2}}\left( \partial_3 + g \partial_\tau\right)\right)H_g^{-2}\right),
\ee
$$
\phi_{3\alpha/4} \partial_3 H_{g}^{-2}
= \phi_{\alpha/4} \partial_3 H_g^{-1} \phi_{\alpha/2} H_g^{-1} -
\phi_{\alpha/4} \partial_3 H_g^{-1} \phi_{\alpha/2}'' H_g^{-2} + 2 \phi_{\alpha/4} \partial_3 H_g^{-1} (\phi'_{\alpha/2})^2 \phi_{-\alpha/2} H_g^{-2}
$$
$$
-2 \phi_{\alpha/4} \partial_3 H_g^{-1}\phi'_{\alpha/2} \phi_{-\alpha/2} \left(\partial_3 + g \partial_\tau\right) H_g^{-1} \left(\phi_{\alpha/2} H_g^{-1} - 2
\left(\phi_{\alpha/2}'' + {\phi'_{\alpha/2}} \left( \partial_3 + g \partial_\tau\right)\right)H_g^{-2}\right)
$$
\bel{gr42}
+\phi'_{\alpha/2} \phi_{-\alpha/4}H_g^{-1}\left(
\phi_{\alpha/2} H_g^{-1} - \left(\phi_{\alpha/2}'' +
2\left(\pd_3 + g\pd_\tau\right)\right) H_g^{-2}\right).
    \ee
Bearing in mind that $S_p \subset S_q$ if $p<q$, and that
$H_g^{-1}$ is a bounded operator, we find that Lemma \ref{lgr3}
implies that all the terms at the r.h.s. of \eqref{gr41} and
\eqref{gr42} can be presented either as a product of an operator
in $S_2$ and an operator in $S_4$, or as a product of three
operators in $S_4$, which yields \eqref{gr40}, and the proof of Theorem \ref{t1} is complete.
\end{proof}

\section{Kato theory for a constant twisting }
\label{sect-const-twist}

\noindent In this section we assume that $\theta' = \beta$ is
constant. Then by \eqref{hbeta} the operator $H_\beta$ is
unitarily equivalent to $ \int _\re ^\oplus  h_\beta(k)\, dk$ with
$h_\beta(k) = -\Delta_\omega + (i\beta\partial_\tau - k)^2$, $k
\in \re$. The goal of the section is to establish various
properties of the fiber operator $h_\beta(k)$, which will be used
later in Section \ref{sect-mourre} for the Mourre estimates
involving the commutator $[H_\beta, iA]$ with a suitable conjugate
operator $A$ described in Section \ref{sect-conjugate}.

\begin{lemma} \label{type A}
The operators $h_\beta(k)$, $k \in \re$, with common domain $\h^2(\omega)\cap\h^1_0(\omega)$, form a self-adjoint holomorphic family of type (A) in the sense of Kato.
\end{lemma}
\begin{proof}
Note that
$$
h_\beta(k) = h_\beta(0) - 2\, {\rm Re}\, \beta\, i \, k  \partial_\tau +k^2,
$$
and that $h_\beta(0)$ is self-adjoint on $\h^2(\omega)\cap\h^1_0(\omega)$.  Let $u\in \h^1_0(\omega)$. Then for any $\eps>0$ we have
$$
\|\beta\, i\partial_\tau\, u\|_2^2 \leq (u, h_\beta(0)\, u)_{{\rm
L}^2(\omega)} \leq \|u\|_2\, \|h_\beta(0)\, u\|_2 \leq \eps^{-1}
\|u\|_2^2+\eps\,  \|h_\beta(0)\, u\|_2^2.
$$
Hence $\beta i\partial_\tau$ is relatively bounded with respect to
$h_\beta(0)$ with relative bound zero and the assertion follows
from \cite[Theorem VII.2.6]{K}.
\end{proof}

\noindent  From Lemma \ref{type A} and the Rellich Theorem,  \cite[Theorem VII.3.9]{K}, it
follows that all the eigenvalues of $h_\beta(k)$ can be represented by a family of functions
\begin{equation} \label{anal-vect}
\{ \lambda_\ell(k)\}_{\ell\in \el}, \qquad \el \subset \N, \quad k\in\re,
\end{equation}
which are analytic on $\re$. Each eigenvalue $\lambda_\ell(k)$ has a finite multiplicity which is constant in $k\in\re$. Moreover, if $\ell\neq \ell'$, then $\lambda_\ell(k) = \lambda_{\ell'}(k)$ may hold only on a discrete subset of $\re$.

\begin{lemma}
Let $\lambda_\ell(k)$ be one of the analytic eigenvalues \eqref{anal-vect}. Let $k_0\in \re$ be given. Then
\begin{equation}  \label{lambda}
\big | \sqrt{\lambda_\ell(k)}- \sqrt{\lambda_\ell(k_0)}\, \big | \, \leq \,  |k-k_0|, \qquad  \ k\in\re.
\end{equation}
\end{lemma}

\begin{proof}
By \cite[Theorem VII.3.9]{K} there exists an analytic normalized
eigenvector $\psi_\ell(k)$ associated to $\lambda_\ell(k)$. From
the Feynman-Hellmann formula, see e.g. \cite[Section VII.3.4]{K},
we obtain
\begin{align*}
| \partial_k\lambda_\ell(k)|^2 & = 4\, | ((k-i\beta \partial_\tau)\,  \psi_\ell(k), \, \psi_\ell(k))_{{\rm L}^2(\omega)}|^2 \leq 4\, \|(k-i\beta \partial_\tau ) \, \psi_\ell(k)\|_{{\rm L}^2(\omega)}^2 \\
& \leq 4\, (\psi_\ell(k), h_\beta(k)\,  \psi_\ell(k))_{{\rm L}^2(\omega)} = 4\, \lambda_\ell(k),    \qquad  \ k\in\re.
\end{align*}
Hence
\begin{equation} \label{lambda'}
|\partial_k\lambda_\ell(k)| \, \leq \, 2\, \sqrt{\lambda_\ell(k)} \qquad  k\in\re.
\end{equation}
By integrating this differential inequality we arrive at \eqref{lambda}.
\end{proof}

\begin{remark}
The eigenvalues $E_n(k)$ given in \eqref{En} might be degenerate. For example if $ \beta =0$ and if the operator $-\Delta_\omega$ has a degenerate eigenvalue $
\mu_n=\mu_m= \mu $, then  $E_n (k) = E_m(k)= \mu^2 + k^2, \; \forall k \in \re$.

\noindent On the other hand, since every $E_n(k)$ coincides with one of the functions $\lambda_\ell(k)$ locally on intervals between the crossing points of $\{ \lambda_\ell(k)\}_{\ell}$, its multiplicity on these intervals is constant.
\end{remark}

\noindent Let us define the set
\begin{align*}
\E_c  := &  \{ E \in \re\, : \, \exists\, k \in \re ,\,
\exists\, \ell,\ell'\in\el,\ \ell\neq \ell' \, :  \lambda_\ell(k) = \lambda_{\ell'}(k)  =E \} \\
& \cup \{E\in \re\, :\exists\, k\in\re,\, \exists\, \ell\in\el\, :\,  \lambda_\ell(k)=E \ \wedge \ \partial_k \lambda_\ell(k) =0\}.
\end{align*}

\begin{lemma} \label{l4}
Let $R\in\re$. Then the set $(-\infty,R] \cap \E_c$ is finite. Moreover, there exists an $N_R \in \mathbb N$ such that for all $n > N_R$ and all $k \in \mathbb R$ we have $E_n(k) > R$.
\end{lemma}

\begin{proof}
We know that
\begin{equation} \label{bkrs}
\inf \sigma(h_\beta(k)) = E_1(k) \, \geq \, E_1(0)  +  c \, k^2,   \qquad  \ k \in \re,
\end{equation}
for some $c\in(0,1)$, see \cite[Theorem 3.1]{BrKoRaSo}.  This
means that there exists some $k_R >0$ such that
\begin{equation} \label{kr}
E_1(k) > R, \qquad \ k\, : |k| > k_R.
\end{equation}
Let us denote $I_R= [-k_R, k_R]$. Hence for any $\ell \in \el$ we have $\lambda_\ell(k) \geq E_1(k) >R$ on $\re\setminus I_R$. We claim that the set
$$
\el_R := \{\ell\in\el\, :\, \exists\, k\in I_R\, :\, \lambda_\ell(k) \leq R\,  \}
$$
is finite. Indeed, if  $\# \el_R = \infty$, then, in view of \eqref{kr}, there is an infinite sequence $\{k_j\}
\subset I_R$ such that  $\lambda_j(k_j) =R$ for all $j\in \el_R$.
  By inequalities \eqref{lambda} and \eqref{lambda'} it follows that
$$
\sup_{j\in\el_R}\, \max_{k\in I_R}\,   |\partial_k \lambda_j(k)|
\leq 4\, k_R + 2\sqrt{R}.
$$
Let $k_\infty\in I_R$ be an accumulation point of the sequence
$\{k_j\}$. Hence, for any $\eps>0$ there exists an infinite set
$\J_\eps\subset \el_R$ such that $|\lambda_{j}(k_\infty) -R| \leq
\eps$ for all $j\in\J_\eps$. This means that $R$ is an
accumulation point of the spectrum of $h_\beta(k_\infty)$ which
contradicts the fact that  $\sigma(h_\beta(k_\infty))$ is
discrete. We thus conclude that the set $\el_R$ is finite.

\smallskip

\noindent Since $\lambda_\ell(k)-\lambda_{\ell'}(k)$ is an
analytic function for any  $\ell,\ell'\in \el_R$, it has finitely
many zeros in the interval $I_R$. Next, by \eqref{bkrs} it follows
that none of the eigenvalues $\lambda_\ell(k), \ell\in \el$, is
constant and therefore, by analyticity, every $\partial_k
\lambda_\ell(k)$ has finitely many zeros in $I_R$. Hence the sets
$$
\cup_{\ell\neq \ell', \ell, \ell' \in \el_R} \{k\in I_R\, :\, \lambda_\ell(k)=\lambda_{\ell'}(k) \} \quad \text{and} \quad
\cup_{\ell\in \el_R} \{k\in I_R\, :\, \partial_k\lambda_\ell(k)=  0\}
$$
are finite and therefore $ (-\infty,R]\cap\E_c$ is finite too. As for the second statement of the Lemma, note that, by \eqref{kr}, $E_n(k) >R$ for all $k\not\in I_R$ and for all $n\in\N$. If we now set
$N_R = \# \el_R+1$, then $N_R$ satisfies the claim.
\end{proof}
\begin{proof}[Proof of  Lemma \ref{loc-finite}]
Let $-\infty < a < b < \infty$ be given. By Lemma \ref{l4} we know
that $\E_c \cap (a,b)$ is a finite set. Since the functions
$E_n(k)$ are analytic away from the crossing points of the
functions \eqref{anal-vect}, it follows that $\E \subset \E_c$.
Hence $\E \cap (a,b)$ is finite too.
\end{proof}

\begin{lemma} \label{l2} Let $I\subset \re$ be an open interval. Assume that $E_n(k)$ is analytic on $I$ and let $p_n(k)$ be the associated eigenprojection. Then
\begin{equation} \label{FH}
 p_n(k)\, \partial_k E_n(k) = 2\,  p_n(k)\, (k-i\beta \partial_{\tau} )\, p_n(k), \qquad  \ k\in I.
\end{equation}
\end{lemma}

\begin{proof}
Since $E_n(k)$ is analytic on $I$, it coincides there with one of
the analytic functions \eqref{anal-vect}. Hence by the Rellich
Theorem, \cite[Theorem VII.3.9]{K}, there exists a family of
orthonormal eigenvectors $\phi_n^j(k), j=1,\dots,q(n,I)$, analytic
on $I$, associated with $E_n(k)$. Here $q(n,I)$ denotes the
multiplicity of $E_n(k)$ on $I$. Since $(\phi_n^j(k), \,
\phi_n^i(k) )_{{\rm L}^2(\omega)}=\delta_{i,j}$ for all $k\in I$,
where $\delta_{i,j}$ is the Kronecker symbol, we have
$$
( h_\beta(k) \, \phi_n^j(k), \, \phi_n^i(k) )_{{\rm L}^2(\omega)} = E_n(k)\, \delta_{i,j} \qquad k\in I.
$$
By differentiating this identity with respect to $k$, we easily obtain
\begin{equation} \label{fh-eq}
2 \, ( (k-i\beta \partial_{\tau} )\, \phi_n^j(k), \phi_n^i(k))_{{\rm L}^2(\omega)}= \partial_k E_n(k)\,\delta_{i,j} \qquad   k\in\re,
\end{equation}
Hence for any $u\in {\rm L}^2(\omega)$
\begin{align*}
2\, p_n(k)\, (k-i\beta  \partial_{\tau})\,  p_n(k)\, u
& = 2 \sum_{i,j=1}^{q(n,I)}\, \phi_n^i(k)\, (\phi_n^i(k),(k-i\beta \partial_{\tau} )\, \phi_n^j(k))_{{\rm L}^2(\omega)}\,  ( \phi_n^j(k), u)_{{\rm L}^2(\omega)} \\
& = \partial_k E_n(k) \, \sum_{j=1}^{q(n,I)}\, \phi_n^j(k)\,  ( \phi_n^j(k), u)_{{\rm L}^2(\omega)}  = \partial_k E_n(k) \, p_n(k)\, u.
\end{align*}
\end{proof}

\medskip

\noindent For the next lemma we need the following definition. Let $\mathcal{I}\subset \re $ be an open interval. Fix $0<\eta < |\mathcal{I}|/2$ and define the interval
\begin{equation}  \label{eps-int}
\mathcal{I}(\eta) := \{ r \in \mathcal{I}  \; : \; \rm{dist}(r ,\re
\setminus \mathcal{I}) \geq \eta \}.
\end{equation}
Let $\chi_\mathcal{I}$ be a $C^\infty$ smooth function such that
\begin{equation} \label{chi}
\chi_\mathcal{I}(r) =1 \quad  \text{if} \quad r \in \mathcal{I}(\eta) \qquad  \text{and} \qquad \chi_\mathcal{I}(r) =0 \quad \text{if}\quad r \notin  \mathcal{I}.
\end{equation}

\begin{lemma} \label{l3} Suppose that $I\subset \re$ is an open interval. Let $\lambda(k)$ and $\mu(k)$ be two analytic functions from the family \eqref{anal-vect} and assume that there is exactly one point $k_0\in I$ such that $\lambda(k_0)=\mu(k_0)$, and $\lambda(k) \neq \mu(k)$ for $k_0\neq k \in I$. Let $\pi_\lambda(k)$ and $\pi_\mu(k)$ be the eigenprojections associated with $\lambda(k)$ and $\mu(k)$.  Then in the sense of quadratic forms on ${\rm L}^2(\omega)$ we have
\begin{align} \label{l31}
& \chi_\mathcal{I}(\lambda(k))\, \pi_\lambda(k)\, (k-i\beta \partial_{\tau} )\, \pi_\mu(k)\, \chi_\mathcal{I}(\mu(k)) \, \leq
\\
 & \qquad \qquad \qquad \leq  b_{\lambda,\mu} \, |\lambda(k)-\mu(k)|\, (\chi^2_\mathcal{I}(\lambda(k))\, \pi_\lambda(k)+\chi^2_\mathcal{I}(\mu(k))\, \pi_\mu(k)) \nonumber
\end{align}
for all $k\in I, k\neq k_0$, where $b_{\lambda,\mu}>0$ is a constant which depends only on $\lambda,\mu$ and $I$.
\end{lemma}
\begin{proof}
Let  $q(\lambda)$ and $q(\mu)$ denote the multiplicities of
$\lambda(k)$ and $\mu(k)$. Let $\psi_\lambda^i(k),
i=1,\dots,q(\lambda)$ and $\psi_\mu^j(k), i=1,\dots,q(\mu)$ be
sets of mutually orthonormal eigenvectors associated to
$\lambda(k)$ and $\mu(k)$. By the Rellich Theorem, \cite[Theorem
VII.3.9]{K}, these vectors can be chosen analytic in $k$. Hence,
by differentiating the equation
$$
( h_\beta(k) \, \psi_\lambda^i(k), \, \psi_\mu^j(k) )_{{\rm L}^2(\omega)} =0 \qquad \ k\in I,\ k\neq k_0
$$
with respect to $k$ we arrive at
\begin{equation} \label{lmu}
2( (k-i\beta \partial_\tau) \, \psi_\lambda^i(k), \, \psi_\mu^j(k) )_{{\rm L}^2(\omega)} = (\lambda(k)-\mu(k)) (\partial_k \psi_\lambda^i(k), \psi_\mu^j(k))_{{\rm L}^2(\omega)} \quad \ k\in I, \ k\neq k_0.
\end{equation}
Note that for all $k\neq k_0, k\in I$ we have
$$
\pi_\lambda(k) = \sum_{j=1}^{q(\lambda)} \psi_\lambda^j(k)\, (\psi_\lambda^j(k), \cdot\, )_{{\rm L}^2(\omega)},
 \quad \pi_\mu(k) = \sum_{i=1}^{q(\mu)} \psi_\mu^i(k)\, (\psi_\mu^i(k), \cdot\, )_{{\rm L}^2(\omega)}.
$$
Let $u\in {\rm L}^2(\omega)$ and let
\begin{equation} \label{btilde}
\max_{1\leq j \leq q(\mu)} \max_{1\leq i \leq q(\lambda)}\sup_{k\in I} |(\partial_k \psi_\lambda^i(k), \psi_\mu^j(k))_{{\rm L}^2(\omega)}| =: \tilde b_{\lambda,\mu}.
\end{equation}
From \eqref{lmu}  we obtain
\begin{align*}
& (u, \, \chi_\mathcal{I}(\lambda(k))\, \pi_\lambda(k)\, (k-i\beta \partial_{\tau} )\, \pi_\mu(k)\, \chi_\mathcal{I}(\mu(k))\, u )_{{\rm L}^2(\omega)}= \\
& = \frac 12\,  \sum_{j=1}^{q(\lambda)}\, \sum_{i=1}^{q(\mu)}\,  \chi_\mathcal{I}(\lambda(k))\,  \chi_\mathcal{I}(\mu(k))(u, \psi_\lambda^i(k))\,  (\psi_\mu^j(k), u)\, (\lambda(k)-\mu(k)) (\partial_k \psi_\lambda^j(k), \psi_\mu^i(k))\\
& \quad \leq \tilde b_{\lambda,\mu}\, |\lambda(k)-\mu(k)|\,  \sum_{j=1}^{q(\lambda)}\, \sum_{i=1}^{q(\mu)}\, \big( \chi^2_\mathcal{I}(\lambda(k)) |(u, \psi_\lambda^j(k))|^2+\chi^2_\mathcal{I}(\mu(k))\, |(\psi_\mu^i(k), u)|^2 \big) \\
& \quad \leq  b_{\lambda,\mu}\,  |\lambda(k)-\mu(k)|\, \big( \chi^2_\mathcal{I}(\lambda(k)) (u, \pi_\lambda(k)\, u)+ \chi^2_\mathcal{I}(\mu(k))\, (u,\pi_\mu(k)\, u) \big),
\end{align*}
for all $k\neq k_0, k\in I$,
where $b_{\lambda,\mu} = \tilde b_{\lambda,\mu}\, \max\{q(\lambda), q(\mu)\}$.
\end{proof}

\smallskip

\section{The conjugate operator}
\label{sect-conjugate} \noindent This section is devoted to the
construction of the conjugate operator $A$ occurring in the Mourre
estimates
obtained in the subsequent two sections. \\

\noindent Pick $\gamma \in C_0^\infty(\re; \re)$, and introduce
the operator
\begin {equation} \label{A}
\hat{A}_0 = \frac{i}{2 } (\gamma\, \partial_k + \partial_k\,
\gamma), \quad {\rm D} (\hat{A}_0)=\mathcal{S}(\re),
\end{equation}
with $\mathcal{S}(\re)$ being the Schwartz class on $\re$.

\begin{prop} \label{essa}
Let $\gamma\in C_0^\infty(\re; \re)$. Then the operator
$\hat{A}_0$ defined in \eqref{A} is essentially self-adjoint in
${\rm L}^2(\re)$.
\end{prop}

\begin{proof}
Without loss of generality we may assume that there exist $a<b$
such that $\gamma(a)=\gamma(b)=0$ and $\gamma(k)>0$ for
$k\in(a,b)$. Consider solutions $u_{\pm}$ to the equations
\begin{equation}
(\hat{A}_0^*\, u)(k) = \frac{i}{2}\,  (\gamma(k) \partial_k + \partial_k \, \gamma(k))\, u(k) = \pm\, i\, u(k).
\end{equation}
A direct calculation gives
\begin{equation}  \label{solution}
u_{\pm}(k) =  \exp\Big( \int_{k_0}^k\, \frac{\pm 2-\gamma'(r)}{2\gamma(r)}\, dr\Big), \quad  k \in (a,b)
\end{equation}
for some $k_0 \in (a, b)$. The positivity of $\gamma$ in $(a,b)$
implies that $\gamma'(a)\geq 0$ and $\gamma'(b)\leq 0$. Hence by
the Taylor expansion there exists an $\eps>0$ and positive
constants $d_a, d_b$ such that
$$
\gamma(r) \leq d_a\, (r-a) \quad \text{for\quad } r\in (a,a+\eps), \qquad \gamma(r) \leq d_b\, (b-r) \quad \text{for\quad } r\in (b-\eps,b).
$$
This, combined with \eqref{solution}, yields
\begin{align*}
u_+(k) & = \Big(\frac{\gamma(k_0)}{\gamma(k)}\Big)^{1/2}\, \exp\Big( \int_{k_0}^k \frac{dr}{\gamma(r)} \Big) \geq \Big(\frac{\gamma(k_0)}{\gamma(k)}\Big)^{1/2}\, \exp\Big( \int_{b-\eps}^k \frac{dr}{ d_b(b-r)} \Big) \\
& \geq c_\eps \, (b-k)^{-\frac 12-\frac{1}{d_b}}, \qquad\ k\in (b-\eps, b),
\end{align*}
for some $c_\eps >0$. Hence $u_+ \not\in {\rm L}^2(\re)$. The same argument shows that
$$
u_-(k) \geq \tilde c_\eps \, (k-a)^{-\frac 12-\frac{1}{ d_a}}, \qquad \forall\ k\in (a, a+\eps), \quad \tilde c_\eps>0,
$$
which implies  $u_-\not\in {\rm L}^2(\re)$. We thus conclude that $\hat{A}_0$ has deficiency indices $(0,0)$ and therefore is essentially self-adjoint.
\end{proof}

\noindent We define the self-adjoint operator $\hat{A}$ as the
closure of $\hat{A}_0$ in ${\rm L}^2(\re)$.\\

\noindent Further, we describe explicitly the action of the
unitary group generated by $\hat{A}$.

\noindent Given a $k\in\re$ and a function $\gamma\in C_0^\infty(\re)$, we consider the
initial value problem
\begin{equation} \label{cauchy}
\frac{d}{dt}\, \varphi(t,k) = -\gamma(\varphi(t,k)),\qquad
\varphi(0,k)= k.
\end{equation}

\begin{prop} \label{group-fiber}
The mapping
\begin{equation}\label{wt}
(W(t)f)(k) = |\partial_k\varphi(t,k)|^{1/2}\, f(\varphi(t,k))
\end{equation}
defines a strongly continuous one-parameter unitary group on
${\rm L}^2(\re)$. Moreover, $\hat A$ is the generator of $W(t)$.
\end{prop}

\begin{proof}
Since $\gamma$ is globally Lipschitz, the Cauchy problem \eqref{cauchy} has a unique
global solution. By the regularity of $\gamma$ and
\cite[Corollary V.4.1]{hart}, it follows that $\varphi \in
C^\infty(\re^2)$. Moreover,
\begin{equation} \label{dk}
\partial_k \varphi (t,k) = \exp \Big(-\int_0^t \gamma'(\varphi(s,k))\, ds\Big) \qquad \forall\ t\geq 0,\ \forall\ k\in\re,
\end{equation}
\cite[Corollary V.3.1]{hart}. Hence $\partial_k\varphi(t,k) > 0$.
Since $\varphi(t+t',k) = \varphi(t, \varphi(t',k))$, we have
$$
W(t) \, W(t') = W(t+t').
$$
Next, from \eqref{cauchy} and \eqref{dk} we deduce that for
$k\not\in\, $supp$\, \gamma$ we have $\varphi(t,k)= k$ for all
$t\geq 0$. In order to verify that $W(t)$ is strongly continuous
on ${\rm L}^2(\re)$, let that $f\in {\rm L}^2(\re)$. We then
have
\begin{align} \label{strong-continuous}
\|W(t) f -f\|_{{\rm L}^2(\re)}^2 & \leq 2\, \|\partial_k\varphi(t,k)^{1/2}(f\circ
\varphi(t, k)-f)\|_{{\rm L}^2(\re)}^2 +2\, \|(\partial_k\varphi(t,k)^{1/2}-1)\,
f\|_{{\rm L}^2(\re)}^2 \\
& \leq c \int_{\mbox{\tiny supp}\, \gamma} \big
(|f(\varphi(t,k))-f(k)|^2\, + |\partial_k\varphi(t,k)^{1/2}-1|^2\,
|f(k)|^2\big)\, dk . \nonumber
\end{align}
From \eqref{dk} and from the fact that $\gamma' \in {\rm L}^\infty(\re)$ it is easily seen that
$\varphi(t,k)\to k$ and $\partial_k\varphi(t,k) \to 1$ as $t\to 0$
uniformly in $k$ on compact subsets of $\re$. Since ${\rm supp}\, \gamma$ is compact, \eqref{strong-continuous}
implies that
$$
\|W(t) f -f\|_{{\rm L}^2(\re)} \ \to \ 0, \qquad  \ t \to 0.
$$
Moreover, using \eqref{cauchy}, a direct calculation gives
$$
\frac{d}{dt} (W(t)f)(k)\, \big |_{t=0} = -\frac 12 \, \gamma'(k) f(k)
- \gamma(k) f'(k) = (i\, \hat A\, f)(k), \quad f\in  \mathcal{S}(\re).
$$
Hence by
\cite[Theorem VIII.10]{RSI} it follows that $\hat A$ generates the unitary group $W(t)$.
\end{proof}

\noindent Let $\gamma$ be as in Theorem \ref{thm1}. By Proposition
\ref{essa} and \cite[Theorem VIII.33]{RSI} it follows that the
operator $\id_\omega \otimes   \hat A$ is essentially self-adjoint
on $C_0^\infty(\omega)\otimes \mathcal{S}(\re)$. The same is true
for the operator $\F^*\, ( \id_\omega \otimes   \hat A) \,  \F$.
We define the conjugate operator $A$ in ${\rm L}^2(\Omega)$ as its
closure:
\begin{equation} \label{bA}
A = \bar{A}_0, \qquad
A_0=  \F^*\, ( \id_\omega \otimes   \hat A) \,  \F , \quad {\rm D}(A_0)=  C_0^\infty(\omega)\otimes \mathcal{S}(\re).
\end{equation}
Let $\Gamma$ be the operator in ${\rm L}^2(\re)$ acting as
\begin {equation} \label{Gamma}
(\Gamma\, \psi)(x_3) : =   (2\pi)^{-1/2} \int_\re
\hat\gamma(x_3-t)\, \psi (t)\, dt, \qquad \hat\gamma := \F_1 \gamma,
\end{equation}
where $\F_1$ denotes the Fourier transform from ${\rm L}^2(\re)$ onto ${\rm L}^2(\re)$:
$$
(\F_1\, f)(k)  = (2\pi)^{-1/2} \int_\re e^{-ik s} f(s)\, ds, \qquad f\in {\rm L}^2(\re).
$$
A direct calculation then shows that
$$
A_0 =  -\frac 12\, \id_\omega\otimes (\Gamma \, x_3+ x_3 \, \Gamma ).
$$

\section{Mourre estimates for a constant twisting} \label{sect-mourre}

\noindent In this section we establish a Mourre estimate for the commutator
$[H_\beta, iA]$ with $\beta$ constant and $A$ defined in
\eqref{bA}.

 \noindent In the sequel, we use the following
notation. Given a self-adjoint positive operator $S$, invertible
in ${\rm L}^2(\Omega)$, we denote by ${\rm D}(S^\nu)^*,\,  \nu >0$,
the completion of ${\rm L}^2(\Omega)$ with respect to the norm $\|
S^{-\nu}\, u\|_{{\rm L}^2(\Omega)}$.


\begin{lemma}\label{com}
The commutator $[ \hat H_\beta,\, i(\id_\omega\otimes \hat A) ]$ defined as a quadratic form
on $C_0^\infty(\Omega)$ extends to a
bounded operator from ${\rm D}(\hat H_\beta)$ into ${\rm D}(\hat H_\beta)^*$. Moreover,
\begin{equation} \label{comm}
[ \hat H_\beta,\, i(\id_\omega\otimes \hat A) ] = 2 \gamma(k) ( k-i\beta \partial_\tau).
\end{equation}
\end{lemma}

\begin{proof}
Let $u \in C_0^\infty(\Omega)$. A simple calculation then gives
$$
(\hat H_\beta\, u,\, i (\id_\omega\otimes \hat A)\, u)_{{\rm L}^2(\Omega)} -
(i (\id_\omega\otimes \hat A)\, u,\, \hat H_\beta\, u)_{{\rm L}^2(\Omega)}= 2(u, \gamma\,
(k-i\beta \partial_\tau)\, u)_{{\rm L}^2(\Omega)}.
$$
Hence \eqref{comm} follows. Moreover, from the above equation we easily obtain
\begin{equation} \label{commut-bound}
 \vert ([ \hat H_\beta,\, i(\id_\omega\otimes \hat A)]\, u,\, u)_{{\rm L}^2(\Omega)} \vert \leq \, C\, \big(
\|\hat H_\beta^{1/2}\, u\|_{{\rm L}^2(\Omega)}^2 +\|u\|_{{\rm L}^2(\Omega)}^2\big).
\end{equation}
So $[ \hat H_\beta, i (\id_\omega\otimes \hat A)]$ is
a bounded operator from ${\rm D}(\hat H_\beta^{1/2})$ into ${\rm D}(\hat H_\beta^{1/2})^*$, and hence it is also bounded from ${\rm D}(\hat H_\beta)$ into ${\rm D}(\hat H_\beta)^*$.
\end{proof}

\medskip

\begin{theorem} \label{thm1}
Let  $E \in \re \setminus {\cal
E}$. Then there exist $\delta  >0$, a function $\gamma \in
C_0^\infty(\re)$ and a positive constant $c=c(E,\delta)$ such that in
the form sense on ${\rm L}^2(\Omega)$ we have
\begin {equation} \label{comm 2}
\chi_\mathcal{I}( \hat H_\beta)\, [ \hat H_\beta, i (\id_\omega \otimes \hat A) ]\, \chi_\mathcal{I}( \hat H_\beta) \geq
c\, \chi_\mathcal{I} ^2 (\hat H_\beta),
\end{equation}
where $\mathcal{I} = (E-\delta,  E + \delta)$,  $\chi_\mathcal{I}$  is given by \eqref{chi} and the commutator $[ \hat H_\beta, i (\id_\omega\otimes \hat A) ]$ is understood as a bounded operator from ${\rm D}(\hat H_\beta)$ into ${\rm D}(\hat H_\beta)^*$.
\end{theorem}
\begin{proof}
First of all we chose $\delta$ small enough such that
\begin{equation} \label{dist}
{\rm dist} (E, (\E_c\setminus E) ) > \delta,
\end{equation}
which is possible in view of Lemma \ref{l4}. Recall that $\E\subset \E_c$. Next we define
$$
\mathcal{K}(n,E) = \big\{k\in\re\, :\, E_n(k)=E\big\} .
$$
Note that by Lemma \ref{l4} $\K(n,E)$ is finite for every $n$ and $\K(n,E)=\emptyset$ for all $n>N_{E+\delta}$. In the rest of the proof we use the notation $N = N_{E+\delta}$.  Let
$$
\K(E) = \cup_{n=1}^\infty\ \K(n,E) = \cup_{n=1}^{N}\ \K(n,E),
$$
and define
\begin{align*}
\K_0(E) & = \{k\in\re\,: \text{there exists a unique} \ n\ \text{such that}\  E_n(k) = E\}, \\
\K_1(E) & = \K(E)\setminus \K_0(E).
\end{align*}
Now we introduce the sets
$$
B(n,E) = \{k\in \re\, : \, E_n(k) \in (E-\delta, E+\delta) \}.
$$
By Lemma \ref{l4} we have $B(n,E) = \emptyset$ for all $n > N$.  From \eqref{dist}  it follows that each $B(n,E)$ is given by a union of finitely many non-degenerate disjoint open intervals:
$$
B(n,E)=  \cup_{j=1}^{G_n}\, Q(j,n), \qquad Q(j,n) \cap Q(i,n) = \emptyset \quad \text{if} \quad i\neq j.
$$
Moreover, every $Q(j,n)$ contains exactly one element of $\K(n,E)$. We will label the intervals $Q(n,E)$ as follows:
\begin{align*}
Q_0(j,n) & := Q(j,n) \quad \text{if} \quad  Q(j,n) \cap \K(n,E)  \subset \K_0(E)\\
Q_1(j,n) & := Q(j,n) \quad \text{if} \quad  Q(j,n) \cap \K(n,E)  \subset \K_1(E).
\end{align*}
By the hypothesis on $E$ we can take $\delta$ small enough so that
$$
Q_0(j,n) \cap Q_0(j,m) = \emptyset \qquad   n\neq m,
$$
and at the same time
$$
Q_1(j,n) \cap \K_1(E) \neq Q_1(i,m) \cap \K_1(E)
$$
implies
$$
Q_1(j,n)\cap Q_1(i,m) = \emptyset.
$$
Hence, for $\delta$ sufficiently small, we can construct intervals $J_{0,l}$ with $l=1,\dots, L(E)$, and $J_{1,p}$ with $p=1,\dots, P(E)$, such that
\begin{equation} \label{disj-1}
J_{0,l} \cap J_{0,l'} =\emptyset \quad  l\neq l', \quad J_{1,p} \cap J_{1,p'} =\emptyset \quad  p\neq p', \quad J_{0,l} \cap J_{1,p} =\emptyset \quad \forall\ l, \, p ,
\end{equation}
and such that
$$
\M_0(E):= \bigcup_{n=1}^N \Big(\cup_j\, Q_0(j,n)\Big) = \bigcup_{l=1}^{L(E)}\, J_{0,l}, \quad
\M_1(E):= \bigcup_{n=1}^N \Big(\cup_j\, Q_1(j,n)\Big) = \bigcup_{p=1}^{P(E)}\, J_{1,p}.
$$
Moreover, each $J_{0,l}$ contains exactly one element $k_{0,l}$ of $\K_0(E)$ and each $J_{1,p}$ contains exactly one element $k_{1,p}$ of $\K_1(E)$.  By construction, we have
$$
\M(E):= \M_0(E) \cup \M_1(E) = \cup_{n=1}^N B(n,E), \qquad \M_0(E)\cap \M_1(E)=\emptyset.
$$
With these preliminaries  we can proceed with the estimation of the commutator. From Lemma \ref{com} we find that
\begin {align} \label{com-eq1}
 \chi_\mathcal{I}( \hat H_\beta)\,  [ \hat H_\beta, &\,  i(\id_\omega\otimes \hat A)]\,  \chi_\mathcal{I}( \hat H_\beta)= \\
&=2\! \sum_{n,m=1}^\infty  \int_\re^\oplus \chi_\mathcal{I} (E_n(k)) p_n(k)
\gamma(k)(k-i\beta \partial_\tau) \chi_\mathcal{I} (E_m(k))p_m(k)\, dk \nonumber \\
& = 2\! \sum_{n,m=1}^N  \int_{\M_0(E)}^\oplus \chi_\mathcal{I} (E_n(k)) p_n(k)
\gamma(k)(k-i\beta \partial_\tau) \chi_\mathcal{I} (E_m(k))p_m(k)\, dk \nonumber \\
&  \quad +2\! \sum_{n,m=1}^N  \int_{\M_1(E)}^\oplus \chi_\mathcal{I} (E_n(k)) p_n(k)
\gamma(k)(k-i\beta \partial_\tau) \chi_\mathcal{I} (E_m(k))p_m(k)\, dk .
 \nonumber
\end{align}
To estimate the first term on the right hand side of \eqref{com-eq1} we note that by construction of $\M_0(E)$, for each $l=1,\dots, L(E)$ there exists exactly one $n(l)\leq N$ such that
\begin{align*}
 \sum_{n,m=1}^N  \int_{J_{0,l}}^\oplus & \chi_\mathcal{I} (E_n(k)) p_n(k)
\gamma(k)(k-i\beta \partial_\tau) \chi_\mathcal{I} (E_m(k))p_m(k)\, dk =\\
& =  \int_{J_{0,l}}^\oplus \chi_\mathcal{I} (E_{n(l)}(k)) p_{n(l)}(k)
\gamma(k)(k-i\beta \partial_\tau) \chi_\mathcal{I} (E_{n(l)}(k))p_{n(l)}(k)\, dk.
\end{align*}
Moreover, since $E_{n(l)}(k)$ does not cross any other eigenvalue of $h_\beta(k)$ on $J_{0,l}$, it is analytic on $J_{0,l}$. Hence by Lemma \ref{l2} we obtain
\begin{align*}
 \int_{J_{0,l}}^\oplus \chi_\mathcal{I} (E_{n(l)}(k)) p_{n(l)}(k)
& \gamma(k)(k-i\beta \partial_\tau) \chi_\mathcal{I} (E_{n(l)}(k))p_{n(l)}(k)\, dk =\\
&  = \int_{J_{0,l}}^\oplus \chi^2_\mathcal{I}(E_{n(l)}(k))\,  p_{n(l)}(k)\,
\gamma(k)\, \partial_k E_{n(l)} dk.
\end{align*}
In view of \eqref{disj-1} we can choose the function $\gamma$ such that
\begin{equation} \label{gamma-1}
\gamma(k)\, \partial_k E_{n(l)}(k) = |\partial_k E_{n(l)}(k)| \qquad \forall \ k\in J_{0,l}\, , \quad \forall\ l=1,\dots, L(E).
\end{equation}
Note that $|\partial_k E_{n(l)}(k)|$ is strictly positive on $J_{0,l}$. Therefore we have
$$
d_0 := \min_{1\leq l \leq L(E)} \, \inf_{k\in J_{0,l}}\, |\partial_k E_{n(l)}(k)| >0.
$$
Hence,
\begin{align} \label{m0}
 \sum_{n,m=1}^N  \int_{\M_0(E)}^\oplus \chi_\mathcal{I} (E_n(k))\,  p_n(k)\,
\gamma(k) & (k-i\beta \partial_\tau)\,  \chi_\mathcal{I} (E_m(k))p_m(k)\, dk \geq  \\
& \geq d_0 \sum_{n=1}^N  \int_{\M_0(E)}^\oplus \chi_\mathcal{I}^2(E_n(k)) p_n(k)\, dk.
 \nonumber
\end{align}
Let us now estimate the second term on the right hand side of \eqref{com-eq1}. On every interval $J_{1,p}$ we have
\begin{align*}
 \sum_{n,m=1}^N  \int_{J_{1,p}}^\oplus & \chi_\mathcal{I} (E_n(k)) p_n(k)
\gamma(k)(k-i\beta \partial_\tau) \chi_\mathcal{I} (E_m(k))p_m(k)\, dk =\\
& =  \int_{J_{1,p}}^\oplus \sum_{r,r'\in R(p)} \chi_\mathcal{I} (E_{r}(k))\, p_{r}(k)\,
\gamma(k)(k-i\beta \partial_\tau) \chi_\mathcal{I} (E_{r'}(k))\, p_{r'}(k)\, dk
\end{align*}
for some $R(p) \subset \{1,\dots, N\}$. Moreover, from the construction of the intervals $J_{1,p}$ it follows that there exists a family of analytic eigenfunctions $\lambda_s(k), s\in S(p)$, with $S(p)$ being a finite subset of $\N$, such that
each $E_r(k)$ coincides with some $\lambda_s(k)$ on $J_{1,p} \cap (-\infty, k_{1,p})$ and with some $\lambda_{s'}(k)$ on $J_{1,p} \cap (k_{1,p}, \infty)$, where $k_{1,p}$ is the only element of $\K_1(E)$ contained in $J_{1,p}$. Let $\pi_s(k)$ be the eigenprojection associated with $\lambda_s(k)$.
With the help of Lemma \ref{l3} we obtain
\begin{align}
 \int_{J_{1,p}}^\oplus & \sum_{r,r'\in R(p)}  \chi_\mathcal{I} (E_{r}(k))\,   p_{r}(k)
\gamma(k)(k-i\beta \partial_\tau) \chi_\mathcal{I} (E_{r'}(k))p_{r'}(k)\, dk = \nonumber  \\
& = \int_{J_{1,p}}^\oplus \sum_{s,s'\in S(p)} \chi_\mathcal{I} (\lambda_{s}(k)) \, \pi_{s}(k)
\gamma(k)(k-i\beta \partial_\tau) \chi_\mathcal{I} (\lambda_{s'}(k))\, \pi_{s'}(k)\, dk = \nonumber  \\
& = \int_{J_{1,p}}^\oplus \sum_{s\in S(p)} \chi^2_\mathcal{I} (\lambda_{s}(k)) \,
\gamma(k)\, \partial_k\lambda_s(k)\, \pi_{s}(k)\, dk +  \nonumber  \\
&\quad  +\int_{J_{1,p}}^\oplus \sum_{s\neq s'\in S(p)} \chi_\mathcal{I} (\lambda_{s}(k)) \, \pi_{s}(k)
\gamma(k)(k-i\beta \partial_\tau) \chi_\mathcal{I} (\lambda_{s'}(k))\, \pi_{s'}(k)\, dk. \label{jp-1}
\end{align}
Since the intervals $J_{1,p}$ are mutually disjoint and also disjoint from the set $\M_0(E)$, see \eqref{disj-1}, and since the functions $\partial_k\lambda_s(k)$ are either all strictly negative or all strictly positive on every interval $J_{1,p}$, by the construction of $J_{1,p}$, we can choose $\gamma$ such that, in addition to \eqref{gamma-1}, it holds
\begin{equation} \label{gamma-2}
\gamma(k)\, \partial_k \lambda_{s}(k) = |\partial_k \lambda_s(k)| \qquad \forall \, k\in J_{1,p}\, , \ \forall\, s\in S(p), \ \forall\,  p=1,\dots, P(E).
\end{equation}
Moreover,
$$
d_1 := \min_{1\leq p\leq P(E)} \, \min_{s\in S(p)}\,  \inf_{k\in J_{1,p}}\,  |\partial_k \lambda_s(k)| >0.
$$
Now, to control the last term in \eqref{jp-1} assume that $s\neq s'$ and let $b_{\lambda_s,\lambda_{s'}}$ be the constant given in Lemma \ref{l3} with $I=J_{1,p}$. Note that $|J_{1,p}|$ decreases as $\delta\to 0$. From the explicit expression for $b_{\lambda_s,\lambda_{s'}}$, see \eqref{btilde}, it is then easily seen that there exists $b_p>0$, independent of $\delta$, such that
$$
\max_{s,s'\in S(p), s\neq s'}\, b_{\lambda_s,\lambda_{s'}} \leq b_p.
$$
Hence, \eqref{jp-1},  in combination with Lemmata \ref{l2} and \ref{l3}, yields
\begin{align}
 \int_{J_{1,p}}^\oplus & \sum_{r,r'\in R(p)}  \chi_\mathcal{I} (E_{r}(k))\,   p_{r}(k)
\gamma(k)(k-i\beta \partial_\tau) \chi_\mathcal{I} (E_{r'}(k))p_{r'}(k)\, dk \geq \nonumber  \\
& \geq (d_1-c_p\, b_p\, \delta) \int_{J_{1,p}}^\oplus \sum_{s\in S(p)} \chi^2_\mathcal{I} (\lambda_{s}(k)) \, \pi_{s}(k)\, dk = \nonumber  \\
&  = (d_1-c_p\, b_p\, \delta) \int_{J_{1,p}}^\oplus \sum_{r\in R(p)} \chi^2_\mathcal{I} (E_{r}(k)) \, p_{r}(k)\, dk,\nonumber
\end{align}
where $c_p>0$ depends only on $p$. Therefore we obtain
\begin{align} \label{m1}
 \sum_{n,m=1}^N  \int_{\M_1(E)}^\oplus & \chi_\mathcal{I} (E_n(k))\,  p_n(k)\,
\gamma(k)  (k-i\beta \partial_\tau)\,  \chi_\mathcal{I} (E_m(k))p_m(k)\, dk \geq  \\
& \geq (d_1-C_E\, \delta)\,  \sum_{n=1}^N  \int_{\M_1(E)}^\oplus \chi_\mathcal{I}^2(E_n(k)) p_n(k)\, dk,
 \nonumber
\end{align}
with $C_E= \max_{1\leq p\leq P(E)} c_p\, b_p$. Taking into account \eqref{m0}, we thus conclude that for $\delta$ small enough there exists some $c>0$ such that
\begin {align}
& \sum_{n,m=1}^N  \int_{\M(E)}^\oplus \chi_\mathcal{I} (E_n(k)) p_n(k)
\gamma(k)(k-i\beta \partial_\tau) \chi_\mathcal{I} (E_m(k))p_m(k)\, dk \nonumber \\
&  \quad \geq c\,  \sum_{n=1}^N  \int_{\M(E)}^\oplus \chi_\mathcal{I}^2(E_n(k)) p_n(k)\, dk = c\,  \sum_{n=1}^\infty  \int_{\re}^\oplus \chi_\mathcal{I}^2(E_n(k)) p_n(k)\, dk = c\, \chi^2_\mathcal{I}(\hat H_\beta).
 \nonumber
\end{align}
In view of \eqref{com-eq1} this proves the theorem.

\end{proof}

\smallskip

\begin{corl} \label{Cthm1}
Let $E \in \re \setminus \E $ and $\mathcal{I} = (E-\delta, E +
\delta)$ be  given as in Theorem \ref{thm1}. Then
\begin {equation} \label{comm 5}
\chi_\mathcal{I}( H_\beta)\, [ H_\beta, iA] \, \chi_\mathcal{I}( H_\beta) \geq c\,
\chi_\mathcal{I} ^2 (H_\beta),
\end{equation}
where $[H_\beta, i A]$ is understood as a bounded operator from
${\rm D}(H_\beta)$ into ${\rm D}(H_\beta)^*$, and the conjugate operator is defined by \eqref{A} and \eqref{bA}.
\end{corl}

\begin{proof}
This follows from \eqref{hbeta}, \eqref{bA} and Theorem \ref{thm1}.
\end{proof}

\smallskip

\section{ Perturbation of the constant twisting}
\label{sect-perturbation}

\subsection{Mourre estimate for $[H_{\theta'}, i A]$}
\label{subsect-mourre-perturbed}

\noindent In the sequel we will suppose that
$$
\theta'(x_3) = \beta -\eps(x_3).
$$
In this section we will prove a Mourre estimate for the commutator $[H_{\theta'}, i A]$, see below Theorem \ref{mourre-estim}.
Notice that $H_{\theta'}$ acts as
\begin{equation} \label{decomp1}
H_{\theta'}= H_\beta + W,  \quad W= (2 \varepsilon
\beta-\varepsilon^2)\partial^2_\tau+ 2 \, \varepsilon\,
\partial_\tau\partial_3 ++ \eps'\, \partial_\tau
\end{equation}
on $\h^1_0(\Omega)\cap \h^2(\Omega)$, cf. Corollary \ref{fgr1}. Together with \eqref{decomp1} we will also use the decomposition
\begin{equation} \label{decomp2}
H_{\theta'}= H_0 + U, \quad U= W -\beta^2\partial^2_\tau- 2 \, \beta\, \partial_\tau\partial_3.
\end{equation}

\begin{lemma}\label{difp}
Let $\chi_\mathcal{I}\in C^\infty_0(\re)$ be given by \eqref{chi}. Then the operator $\chi_\mathcal{I}(H_{\theta'})-\chi_\mathcal{I}(H_\beta) $ is compact  in ${\rm L}^2(\Omega)$.
\end{lemma}
\begin{proof}
The Helffer-Sj\"ostrand formula, \cite{Da, dsj}, gives
\begin{equation}\label{424}
  \chi_\mathcal{I}(H_{\theta'})-\chi_\mathcal{I}(H_\beta)= - \frac{1}{\pi} \int_{\re^2} \frac{\partial
    \tilde{\chi}}{\partial \bar{z}} \, (H_{\theta'}-z)^{-1} W (H_\beta - z)^{-1} \,dx\, dy,
\end{equation}
where $z=x+iy$, and $\tilde{\chi}$ is a compactly supported quasi-analytic
extension of $\chi_\mathcal{I}I$  in $\re^2$ which satisfies
\begin{equation}\label{425}
\sup_{x\in\re}\, \Big|\frac{\partial\tilde{\chi
}}{\partial\overline{z}}(x+i y) \Big|\, \leq \, \mbox{const}\, y^4\;
,\qquad |y|\leq 1 .
\end{equation}
Since  $(H_{\theta'}-z)^{-1} W (H_0 - z)^{-1}$ is compact
whenever $y\neq 0$, see \cite{BrKoRaSo}, it follows that
$\frac{\partial\tilde{\chi}}{\partial \bar{z}} (H_{
\theta'}-z)^{-1} W (H_\beta - z)^{-1}$ is compact for all
$(x,y)\in\re^2$ with $y\neq 0$. Moreover, by the resolvent equation
the norm of $(H_{\theta'}-z)^{-1} W (H_0 - z)^{-1}$ is bounded
by a constant times $y^{-2}$. In view of \eqref{425} the integrand
on the right hand side of \eqref{424} is then uniformly norm-bounded
in $\re^2$ and hence $\chi_\mathcal{I}(H_{\theta'})-\chi_\mathcal{I}(H_\beta)$ is
compact.
\end{proof}

\begin{theorem} \label{mourre-estim}
Let $E \in \re \setminus {\cal E}$ and let $\eps$ satisfy
\eqref{decay-eps}. Choose $\delta>0$ and $\gamma  \in C_0^\infty(\re)$ as
in Theorem \ref{thm1}. Then there exists a positive constant
$c$ and a compact operator $K$ in  ${\rm L}^2(\Omega)$ such that
\begin {equation} \label{comm3}
P_{\mathcal{I}(E,\delta)}\,  [ H_{\theta'}, iA]\,  P_{\mathcal{I}(E,\delta)}  \ \geq \ c\  P_{\mathcal{I}(E,\delta)}^2  + P_{\mathcal{I}(E,\delta)}\, K\, P_{\mathcal{I}(E,\delta)},
\end{equation}
where $ P_{\mathcal{I}(E,\delta)}$ is the spectral projection for the interval
$$
\mathcal{I}(E,\delta):=(E-\delta/2, E + \delta/2),
$$
associated to $H_{\theta'}$.
\end{theorem}
\begin{proof}
Let $\mathcal{I} = (E-\delta, E + \delta)$.
We proceed in several steps. First we show that there exists $c>0$
and a compact operator $K_1$ in  ${\rm L}^2(\Omega)$ such that
\begin {equation} \label{Comm1}
 \chi_\mathcal{I}( H_{\theta'})[ H_{\beta}, iA]
 \chi_\mathcal{I}( H_{\theta'}) \geq c\, \chi_\mathcal{I} ^2 (H_\beta) +K_1.
\end{equation}
We write
\begin{align*}
\chi_\mathcal{I}( H_{\theta'})[ H_{\beta}, iA] \chi_\mathcal{I}( H_{\theta'}) & = \chi_\mathcal{I}( H_{\beta})[ H_{\beta}, iA] \chi_\mathcal{I}(
H_{\beta})  + \chi_\mathcal{I}(H_{\beta})[ H_{\beta}, iA]
(\chi_\mathcal{I}(H_{\theta'})- \chi_\mathcal{I}(H_{\beta})) \\
&  +  (\chi_\mathcal{I}(H_{\theta'}) -\chi_\mathcal{I}(H_{\beta}))[
H_{\beta}, iA] \chi_\mathcal{I}(H_{\theta'}).
\end{align*}
By Corollary \ref{Cthm1} there exist $c>0$ such that
$$
\chi_\mathcal{I}( H_\beta)[ H_\beta, iA] \chi_\mathcal{I}( H_\beta) \geq c \chi_\mathcal{I} ^2 (H_\beta).
$$
It can be verified by a simple calculation that the operator
$\Gamma$ defined in \eqref{Gamma} commutes with $H_\beta$. Hence
$$
\chi_\mathcal{I}(H_{\beta})[ H_{\beta}, iA] (\chi_\mathcal{I}(H_{\theta'})-  \chi_\mathcal{I}(H_{\beta})) = 2 \chi_\mathcal{I}(H_{\beta}) (
i\partial_3+ \beta i \partial_\tau) \Gamma (\chi_\mathcal{I}(H_{\theta'})-  \chi_\mathcal{I}(H_{\beta})).
$$
We know that $ \Gamma (\chi_\mathcal{I}(H_{\theta'})-
\chi_\mathcal{I}(H_{\beta})) $ is compact (see e.g. Lemma
\ref{difp}). The operators $ (H_{\theta'} +1)^{-1}( i\partial_3+
\beta i \partial_\tau) $   and $ \chi_\mathcal{I}(H_{\theta'})(H_{\theta'} +1)$ are bounded so $\chi_\mathcal{I}(H_{\beta})
( i\partial_3+ \beta i \partial_\tau)$ is bounded too and $ K_{11}:=
\chi_\mathcal{I}(H_{\beta}) ( i\partial_3+\beta i
\partial_\tau) \Gamma (\chi_\mathcal{I}(H_{\theta'})-
\chi_\mathcal{I}(H_{\beta}))$ is compact. The same arguments show that
$$
K_{12}:=(\chi_\mathcal{I}(H_{\theta'}) -\chi_\mathcal{I}(H_{\beta}))[ H_{\beta}, iA]
\chi_\mathcal{I}(H_{\theta'}):= 2(\chi_\mathcal{I}(H_{\theta'}) -\chi_\mathcal{I}(H_{\beta}))
\Gamma ( i\partial_3- \beta i \partial_\tau)   \chi_\mathcal{I}(H_{\theta'})
$$
is compact. Putting $K_1 = K_{11} + K_{12}$ concludes the first step
of the proof.

\medskip

Next we consider $ \chi_\mathcal{I}( H_{\theta'})[ W, iA]
\chi_\mathcal{I}( H_{\theta'})$. For the sake of simplicity we now
write $s$ instead of $x_3$. Defining
$$
\eta(s):= 2 \varepsilon(s) \beta-\varepsilon(s)^2
$$
we get
\begin{equation} \label{comm-sum}
[ W, iA] = [\eta\, , iA]\, \partial^2_\tau+ [\epsilon\,
\partial_s, iA]\, \partial_\tau+ [ \partial_s\, \epsilon,\, iA]\, \partial_\tau.
\end{equation}
We first deal with the term $[\eta\, \partial^2_\tau, iA]= i  [
\eta, A] \partial^2_\tau= - \frac{i}{2} [\eta, \Gamma s+ s\Gamma ]\,
\partial_\tau^2$. For an appropriate test function $\phi $ we obtain
\begin{align*}
\sqrt{2\pi}\, \, ([\eta , \Gamma s+ s \Gamma ]\phi)(s)  & =: (T
\phi)(s)= \eta(s) \int_{\re}
\hat\gamma(s-s') s'\phi(s')\, ds'  +  \eta(s) \int_{\re} s  \hat\gamma(s- s') \phi(s')\, ds' \\
&\quad  - \int_{\re}  \hat\gamma(s- s')  \eta(s')s'\phi(s')\, ds'
 -
\int_{\re} s \hat\gamma(s- s') \eta(s')\phi(s')\, ds'.
\end{align*}
Hence $T$ is an integral operator on ${\rm L}^2(\re)$ with the kernel
$$
T(s,s')= \eta(s) \hat\gamma(s- s') s'  +  \eta(s) s \hat\gamma(s-
s') - \hat\gamma(s- s')  \eta(s')s' -  s \hat\gamma(s- s') \eta(s').
$$
To control the $s$-dependence we rewrite the kernel as
\begin{align} \label{tss'}
T(s,s') & = \eta(s) \hat\gamma(s- s') (s' -s)  + 2 \eta(s) \, s\,
\hat\gamma(s- s')  \\
& - 2\hat\gamma(s- s')  \eta(s')\, s' -  (s-s') \hat\gamma(s- s')
\eta(s'). \nonumber
\end{align}
Next we recall  that if $f \in {\rm L}^q(\re)$, $g\in {\rm L}^p(\re)$, $q \in [2,\infty)$, $1/q + 1/p = 1$, then the Hausdorff--Young inequality $\|\hat{g}\|_{{\rm L}^q(\re)} \leq (2\pi)^{\frac{1}{2} - \frac{1}{p}} \|g\|_{{\rm L}^p(\re)}$ and the interpolation result which we already used in the proof of Lemma \ref{lgr3} (see \cite[Theorem 4.1]{sim} or \cite[Section 4.4]{BeLo}), imply that the integral operator with a kernel of the form $f(s)\, g(s-s')$, $s, s' \in \re$, belongs to the class $S_q$, and hence is
 compact on ${\rm L}^2(\re)$.
By  \eqref{decay-eps},  both functions $\eta(s)$
and $s\eta(s)$ are in ${\rm L}^q(\re)$ for $q$ large enough. Since
$\gamma\in C_0^\infty(\re)$, its Fourier transform $\hat\gamma$ is
in the Schwartz class on $\re$ and therefore in any ${\rm L}^p(\re)$ with
$p\geq 1$. Therefore, the operator $[\eta , \Gamma s+
s \Gamma ]$ is compact on ${\rm L}^2(\re)$. In order to ensure the
compactness of $ \chi_\mathcal{I}( H_{\theta'})\, T
\,\partial^2_\tau\, \chi_\mathcal{I}( H_{\theta'}) $ on
${\rm L}^2(\Omega)$, we note that by Corollary \ref{fgr1} and the closed
graph theorem the operators $H_\beta^{-1}\, H_{\theta'}$ and
$H_{\theta'}^{-1} \, H_\beta$ are bounded on ${\rm L}^2(\Omega)$.
Since $H_{\theta'}\, \chi_\mathcal{I}( H_{\theta'})$ is bounded
too, it suffices to prove that
\begin{equation} \label{aux}
H_\beta^{-1}\, T\, \partial^2_\tau\, H_\beta^{-1}
\end{equation}
is compact on ${\rm L}^2(\Omega)$. To this end we point out that $H_\beta
\geq \id_\omega\otimes(-\Delta_\omega) $ and that the operators $\partial_\tau^2\,
(-\Delta_\omega)^{-1}\, $ and $(-\Delta_\omega)^{-1}$ are respectively bounded
and compact on ${\rm L}^2(\omega)$. Hence $(\id_\omega\otimes(-\Delta_\omega))^{-1}\,
T\, \partial^2_\tau\, (\id_\omega\otimes(-\Delta_\omega))^{-1}$ is a product of a
bounded and a compact operator and hence is compact on
${\rm L}^2(\Omega)$. This yields the compactness of the operator
\eqref{aux}.

In the same way we deal with the remaining terms on the right hand
side of \eqref{comm-sum}. As for the the operator
$$
[\partial_\tau \epsilon \partial_s,\, iA]= - \frac{i}{2} \, \,
\partial_\tau\, [ \epsilon  \partial_s, \Gamma s+ s \Gamma ],
$$
with the help of the integration by parts we find that
\begin{align*}
([ \epsilon  \partial_s, \Gamma s+ s \Gamma ] \phi) (s)& =  (2\pi)^{-1/2} (R_1\,
\phi)(s)+ (2\pi)^{-1/2}(R_2\, \phi)(s) \\
& =  (2\pi)^{-1/2}\int_{\re}\, R_1(s,s')\, \phi (s')\, ds'
 +  (2\pi)^{-1/2}\int_{\re}\, R_2(s,s')\, \phi (s')\, ds',
\end{align*}
where the integral kernels $R_1(s,s')$ and $R_2(s,s')$ of the
operators $R_1$ and $R_2$ are given by
\begin{align}
R_1(s,s') & = \epsilon(s)\left( \hat\gamma'(s-s')( s' -s)+
\hat\gamma(s-s')  + 2\, s\, \hat\gamma'(s-s')\right) \label{tss'1}\\
R_2(s,s') & = \epsilon(s')\left(-\hat\gamma'(s-s') s' +
\hat\gamma(s-s') - s\, \hat\gamma(s-s')\right)  + \epsilon'(s')
\left(\hat\gamma(s-s') s' +s \hat\gamma(s-s')\right), \nonumber
\end{align}
and $\hat\gamma'$ denotes the derivative of $\hat\gamma$. As above
we need to write also the kernel $T_2(s,s')$ as a sum of the terms
of the form $f(s)\, g(s-s')$ and $f(s')\, g(s-s')$:
\begin{align}
R_2(s,s') & = \epsilon(s')\left(-2\hat\gamma'(s-s') s' +
\hat\gamma(s-s') - (s-s')\, \hat\gamma(s-s')\right)  \nonumber \\
&  \quad + \epsilon'(s') \left(2\hat\gamma(s-s') s' +(s-s')
\hat\gamma(s-s')\right). \label{tss'2}
\end{align}
Using the assumptions of Theorem \ref{main} and the fact that $\hat\gamma'$
is the Schwartz class on $\re$, we conclude as before that $R_1$ and
$R_2$ are compact on ${\rm L}^2(\re)$ and therefore
$$\chi_\mathcal{I}( H_{\theta'})\, \partial_\tau\, [ \epsilon\,
\partial_s, \Gamma s+ s \Gamma ]\, \chi_\mathcal{I}( H_{\theta'})
$$
is compact on ${\rm L}^2(\Omega)$. The compactness of
$$
\chi_\mathcal{I}( H_{\theta'})\, [ \partial_s \, \epsilon\,
\partial_\tau, iA]\, \chi_\mathcal{I}( H_{\theta'})
$$
follows in a completely analogous way. Hence we obtain
\begin {equation}
\chi_\mathcal{I}( H_{\theta'})[ H_{\theta'}, iA]\, \chi_\mathcal{I}(
H_{\theta'}) \geq c\, \chi_\mathcal{I} ^2 (H_{\theta'}) +K
\end{equation}
where $\mathcal{I} = (E-\delta, E + \delta)$ and $K$ is compact. Now we fix $\eta = \delta/2$ in \eqref{chi}. The statement then follows by multiplying the last inequality from the left and from the right by $P_{\mathcal{I}(E,\delta)}$.
\end{proof}

\subsection{Proof of Theorem \ref{main}}
\label{subsect-sc}

\noindent In order to prove Theorem \ref{main}, we will need, in addition to the Mourre estimate established in Theorem \ref{mourre-estim}, a couple of technical results.
We introduce the norm
$$
\| u\|_{+2,\theta} : = \big( \|H_{\theta'}\, u\|^2_{{\rm L}^2(\Omega)}+ \|u\|^2_{{\rm L}^2(\Omega)}\big)^{1/2}, \qquad u\in \h^2(\Omega)\cap\h^1_0(\Omega),
$$
and recall that if $\eps$ satisfies \eqref{decay-eps}, then in view of Corollary \ref{fgr1}
\begin{equation} \label{norm-eq}
\| u\|_{+2,\theta} \, \asymp \, \| u\|_{+2,0} \, \asymp\, \| u\|_{\h^2(\Omega)}, \qquad  u\in \h^2(\Omega)\cap\h^1_0(\Omega).
\end{equation}

\smallskip

\begin{prop} \label{conditions}
Let $\eps$ satisfy the assumptions of Theorem \ref{main}. Then

\begin{itemize}
\item[(a)] The unitary group $e^{ i t A}$ leaves ${\rm
D}(H_{\theta'})$ invariant. Moreover, for each $u\in  {\rm
D}(H_{\theta'})$, $\sup_{|t|\leq 1} \|e^{ i t A} u\|_{+2,\theta}
<\infty$.

\smallskip

\item[(b)] The operator $B_0 =[H_0, i A]$ defined as a quadratic form on ${\rm D}(A)\cap {\rm D}(H_0)$
is bounded on ${\rm L}^2(\Omega)$.

\smallskip

\item[(c)] The operator $B= [H_{\theta'}, i A]$ defined as a quadratic form on ${\rm D}(A)\cap {\rm D}(H_{\theta'})$
is bounded from ${\rm D}(H_{\theta'})$ into ${\rm D}(H_{\theta'}^{1/2})^*$.

\smallskip

\item[(d)] There is a common core $C$ for $A$ and $H_0$ so that $A$
maps $C$ into $\h^1_0(\Omega)$.

\end{itemize}
\end{prop}

\begin{proof} Note that $H_0 =-\Delta$ and that ${\rm D}(H_{\theta'})={\rm D}(H_0)=\h^1_0(\Omega)\cap \h^2(\Omega)$,
in view of Corollary \ref{fgr1}.
To prove assertion $(a)$, pick $f\in {\rm D}(H_{\theta'})$ and
denote $g=  \F f$. By Lemma \ref{group-fiber},
\begin{equation} \label{u-group}
(e^{ i t A} f)(x) = \F^*  \big[ (\partial_k\varphi(t,k))^{1/2}\,
g(x_\omega, \varphi(t,k))\big] .
\end{equation}
Hence,
\begin{align} \label{h-norm}
\|\Delta\, (e^{ i t A} f)\|_{{\rm L}^2(\Omega)}^2 & = \| e^{ i t A}(\Delta_\omega f) + \partial_3^2 (e^{ i t A} f)\|_{{\rm L}^2(\Omega)}^2 \\
& \leq \|\Delta_\omega f\|_{{\rm L}^2(\Omega)}^2 +  \| k^2\, (\partial_k\varphi(t,k))^{1/2}\,
g(x_\omega, \varphi(t,k)) \|_{{\rm L}^2(\Omega)}^2, \nonumber
\end{align}
where we have used  the fact that $e^{ i t A}: {\rm L}^2(\Omega)\to {\rm L}^2(\Omega)$ and $\F^* :{\rm L}^2(\Omega) \to {\rm L}^2(\Omega)$ are unitary. Assume that supp$\, \gamma \subset [-k_c,k_c]$ for some $k_c>0$. Then $\varphi(t,k) = k$ and $\partial_k\varphi(t,k) = 1$ for all $k$ with $|k| >k_c$ and all $t\geq 0$, see the proof of Lemma \ref{group-fiber}. We thus obtain
\begin{align*}
\| k^2\, & \partial_k\varphi(t,k)\,
g(x_\omega, \varphi(t,k)) \|_{{\rm L}^2(\Omega)}^2  \leq k_c^4 \int_{\omega\times[-k_c,k_c]} \partial_k\varphi(t,k)\,
|g(x_\omega, \varphi(t,k))|^2\, dk\, dx_\omega \\
& + \int_{\Omega}  k^4\, |g(x_\omega, k)|^2\, dk\, dx_\omega \, \leq \, k_c^4 \int_{\Omega}
|g(x_\omega, z)|^2\, dz\, dx_\omega + \int_{\Omega}  k^4\, |g(x_\omega, k)|^2\, dk\, dx_\omega \\& =
k_c^4\, \| f\|^2_{{\rm L}^2(\Omega)} +  \| \partial_3^2 f\|^2_{{\rm L}^2(\Omega)},
\end{align*}
where in the first integral on right hand side we have used the
change of variables $z=\varphi(t,k)$ taking into account that
$\partial_k\varphi(t,k) >0$, see \eqref{dk}. In view of \eqref{norm-eq} and
\eqref{h-norm} we have
$$
\|e^{ i t A} f\|_{+2,0}^2  = \|\Delta\, (e^{ i t A} f)\|_{{\rm L}^2(\Omega)}^2  +\| f\|_{{\rm L}^2(\Omega)}^2 \, \leq \, {\rm const}\, \|f\|_{\h^2(\Omega)}^2.
$$
This implies that $\sup_{|t|\leq 1} \|e^{ i t A} f\|_{+2,\theta} < \infty$, see \eqref{norm-eq}. Moreover,
since $e^{ i t A} f=0$ on $\partial\Omega$,  see \eqref{u-group}, we  find that $e^{ i t A} f \in {\rm D}(H_{\theta'})$. This proves (a).
Next we note that by Lemma \ref{com}
$$
[H_0, iA] = \id_\omega \otimes  \F_1^* (2\, \gamma(k)\, k)\, \F_1
$$
which is a bounded operator on ${\rm L}^2(\Omega)$. This proves (b).

As for assertion (c), note that  $B=[H_\beta,iA] +[W,iA]$. By inequality \eqref{commut-bound} we know
that $(H_\beta+1)^{-1/2} [H_\beta,iA] (H_\beta+1)^{-1}$ is bounded on ${\rm L}^2(\Omega)$. On the other hand,
from the proof of Theorem \ref{mourre-estim} it follows that the same is true for the operator
$(H_{\beta}+1)^{-1/2} [W,iA] (H_{\beta}+1)^{-1}$. Since $(H_\beta+1)^{\alpha}(H_{\theta'}+1)^{-\alpha}$, $\alpha =1/2, 1$, is bounded,  we conclude that $(H_{\theta'}+1)^{-1/2} B(H_{\theta'}+1)^{-1}$ is bounded on ${\rm L}^2(\Omega)$ and (c) follows.
To prove (d) we define $C :=  {\rm D}(-\Delta_\omega) \otimes \mathcal{S}(\re)$.
By definition of $A$,  $C$ is a core for $A$. On the other hand, $C$ is also a core for $H_0$.
Since $\F :C \to C$ is a bijection and since $\hat A : \mathcal{S} \to \mathcal{S}$, it follows that $A: C\to
C\subset \h^1_0(\Omega)$.
\end{proof}

\begin{lemma} \label{double-com}
Let $\eps$ satisfy assumptions of Theorem \ref{main}. Then $(H_{\theta'}+1)^{-1} [B, A]\,
(H_{\theta'}+1)^{-1}$ is bounded as an operator on ${\rm L}^2(\Omega)$.
\end{lemma}

\begin{proof}
Recall that $B=[H_{\theta'}, i A]$. We write
$$
\B= \B_1+\B_2, \quad \text{where}\quad \B_1= [[H_\beta, i A], i A],
\quad \B_2= [[W, i A], i A].
$$
As for the term $\B_1$, a  direct calculation gives
\begin{equation} \label{double-comm}
\B_1  = \F^*\, [[ \hat H_\beta, i(\id_\omega\otimes\hat A)], i(\id_\omega \otimes\hat A)]\, \F
 = \F^* \left(\gamma(k)^2 +\gamma(k) \gamma'(k) ( k-i\beta\,
\partial_\tau)\right) \F.
\end{equation}
Let $u\in {\rm L}^2(\Omega)$. Similarly as in \eqref{commut-bound} we find out that
$$
| (u,\, (k-i\beta\, \pd_\tau)\, u)_{{\rm L}^2(\Omega)} | \, \leq \, (u, (\hat H_\beta+1)\, u)_{{\rm L}^2(\Omega)}.
$$
Since $\gamma$ and $\gamma'$ are bounded, the last inequality implies that
also $(H_\beta+1)^{-1} \B_1 (H_\beta+1)^{-1}$ is bounded. From Proposition
\ref{op-domain} it then follows that
$$
(H_{\theta'}+1)^{-1} \B_1 (H_{\theta'}+1)^{-1}
$$
is bounded too. As for the remaining part of the double commutator,
we first note that in view of \eqref{comm-sum} and of the fact that
the operators $\partial_\tau (H_{\theta'}+1)^{-1}$ and
$\partial_\tau^2(H_{\theta'}+1)^{-1}$ are bounded, it suffices
to show that
\begin{equation} \label{double-c-1d}
\Big[[\eta\, , \F_1^*\hat A\, \F_1]+ [\epsilon', \, \F_1^* \hat A\, \F_1] + [\epsilon\, \partial_s, \F_1^*\hat A\, {\cal
F_1}] ,\, \F_1^*\hat A \, \F_1\Big]
\end{equation}
is a bounded operator on ${\rm L}^2(\re)$. Let $u\in {\rm L}^2(\re)$ and recall that
$$
(\F_1^* \hat A\, \F_1\, u)(s) = -\frac{1}{2\sqrt{2\pi}}\, \Big( \int_\re \hat\gamma(s-s') s'\, u(s')\, ds' +s \int_\re \hat\gamma(s-s')\, u(s')\, ds'\Big).
$$
It will be useful to introduce the shorthands
$$
\hat\gamma_j(r) = r^j\, \hat\gamma(r).
$$
Note that $\hat\gamma_j \in {\cal S}(\re)$ for all $j\in\N$. We have
\begin{align}
-2\sqrt{2\pi}\,  [\eta\, , \F_1^*\hat A\, \F_1] u  & = 2 \eta(s)\int_{\re}
\hat\gamma(s-s') s' u(s')\, ds' +  \eta(s)\int_{\re}
\hat\gamma_1(s-s') u(s')\, ds'   \label{eta-term} \\
& \quad - \int_{\re} \hat\gamma(s-s') s' \eta(s') u(s')\, ds' -
s\int_{\re} \hat\gamma(s-s') \eta(s') u(s')\, ds' \nonumber \\
& =: \sum_{j=1}^4\,
T_j\, u. \nonumber
\end{align}
Accordingly,
\begin{align*}
-\sqrt{2\pi}\,   [T_1,\,   \F_1^*\hat A \, \F_1] u  &= \eta(s)\int_{\re}\int_{\re}
\hat\gamma(s-s') \, s'^2  \hat\gamma(s'-s'')\,  u(s'')\, ds''  ds' \\
& \quad + \eta(s)\int_{\re}\int_{\re}
\hat\gamma(s-s') s'\, s''\,   \hat\gamma (s'-s'')\, u(s'')\, ds''  ds'
\\
&   \quad- \int_{\re} \int_{\re} \hat\gamma(s-s')\,  \eta(s') \,
\hat\gamma (s'-s'')\, s'\, s''\,  u(s'')\, ds''  ds' \\
& \quad -s \int_{\re}\int_{\re} \hat\gamma(s-s')\,  \eta(s') \,
\hat\gamma (s'-s'')\, u(s'')\,  ds''  ds'\\
& =: \sum_{j=1}^4\,
T_{1,j}\, u.
\end{align*}
Note that
$$
s'^2 \, \hat\gamma(s-s')   \hat\gamma(s'-s'') =s^2\, \hat\gamma(s-s')  \hat\gamma(s'-s'') + \hat\gamma_2(s-s') \hat\gamma(s'-s'') -2s\, \hat\gamma_1(s-s')   \hat\gamma(s'-s''),
$$
which implies
\begin{align*}
T_{1,1}\, u(s) & = s^2 \eta(s)\,  \hat\gamma \ast ( \hat\gamma\ast u) + \eta(s)\, \hat\gamma_2 \ast(\hat\gamma\ast u) -2 s\, \eta(s)\, \hat\gamma\ast(\hat\gamma_1\ast u).
\end{align*}
Hence, by a repeated use of the Young inequality
\begin{equation} \label{young}
\|g \ast h\|_p \leq C\, \|g\|_q\, \|h\|_r, \quad \frac 1q +\frac
1r = 1 +\frac 1p,
\end{equation}
with $p=q=2$ and $r=1$, we get
$$
\|T_{1,1}\, u\|_2 \leq C_{1,1} \left(\|s^2\eta\|_\infty \, \|\hat\gamma\|_1^2 + \|\eta\|_\infty
\|\hat\gamma\|_1\, \|\hat\gamma_2\|_1  +\|s\,
\eta\|_\infty\, \|\hat\gamma\|_1\, \|\hat\gamma_1\|_1 \right)\, \|u\|_2,
$$
for some $C_{1,1} < \infty$. Moreover, since
$$
s' s'' \hat\gamma(s-s')   \hat\gamma(s'-s'') = s'^2 \, \hat\gamma(s-s')   \hat\gamma(s'-s'') -s \, \hat\gamma(s-s')   \hat\gamma_1(s'-s'') +\hat\gamma_1(s-s')   \hat\gamma_1(s'-s''),
$$
with the help of \eqref{young} we obtain
$$
\|T_{1,2}\, u\|_2 \leq \|T_{1,1}\, u\|_2 + C_{1,2} \left( \|\eta\|_\infty
\|\hat\gamma_1\|_1^2   +\|s\,
\eta\|_\infty\, \|\hat\gamma\|_1\, \|\hat\gamma_1\|_1 \right)\, \|u\|_2.
$$
As for
$T_{1,3}$, we note that
\begin{align*}
T_{1,3}\, u & =  \hat\gamma \ast (s\eta (\hat\gamma_1\ast
u)) -\hat\gamma \ast (s^2\eta (\hat\gamma \ast u)),
\end{align*}
which, in combination with \eqref{young}, implies
$$
\|T_{1,3}\, u\|_2 \leq C_{1,3} \left( \|s \eta\|_\infty
\|\hat\gamma_1\|_1\, \|\hat\gamma\|_1 + \|s^2 \eta\|_\infty
\|\hat\gamma\|_1\, \|\hat\gamma\|_1\right)\, \|u\|_2.
$$
Next, for $T_{1,4}\, u$ we find
\begin{align*}
T_{1,4}\, u & =  -\hat\gamma_1 \ast (\eta (\hat\gamma_1\ast
u)) +\hat\gamma_1 \ast (s\eta (\hat\gamma \ast u))
-\hat\gamma \ast (s\eta(\hat\gamma_1\ast u))
+ \hat\gamma \ast (s^2\eta(\hat \gamma\ast u)).
\end{align*}
By using again \eqref{young} we get
$$
\|T_{1,4}\, u\|_2 \leq C_{1,4} \left( \| \eta\|_\infty
\|\hat\gamma_1\|_1^2  +2 \|s  \eta\|_\infty
\|\hat\gamma_1\|_1\|\hat\gamma\|_1+  \| s^2\eta\|_\infty
\|\hat\gamma\|_1^2 \right)\, \|u\|_2.
$$
This implies that
$$
 \|[T_1, \F_1^*\hat A\,  \F_1]\, u\|_ 2= \frac 12 \sum_{j=1}^4 \|T_{1,j}\, u\|_ 2 \leq C_1\, \|u\|_2,
$$
for some $C_1<\infty$.  As for the term $[T_2,  \F_1^* \hat A\,  \F_1]\, u$,  we find out that
\begin{align*}
-2\sqrt{2\pi}\,   [T_2,\,   \F_1^*\hat A \, \F_1] u  &= -\eta\,  \hat\gamma_1 \ast
( \hat\gamma_1\ast u) -2\, \eta\, \hat\gamma_2 \ast ( \hat\gamma\ast u)  + 2s\, \eta\, \hat\gamma_1 \ast
( \hat\gamma\ast u) \\
&\quad - 2 \hat\gamma \ast (s\eta (\hat\gamma_1\ast u)) - \hat\gamma_1 \ast (\eta (\hat\gamma_1\ast
u)).
\end{align*}
Hence by \eqref{young}
$$
\|[T_2, \F_1^*\hat A\,  \F_1]\, u\|_ 2 \leq C_2\,
\left( \|\eta\|_\infty ( \|\hat\gamma_1\|_1^2+ \|\hat\gamma_2\|_1\|\hat\gamma\|_1) +
\|s \eta\|_\infty\, \|\hat\gamma_1\|_1\|\hat\gamma\|_1\right)\, \|u\|_2.
$$
Next we consider the last term on the right hand side of \eqref{eta-term}. A direct
calculation gives
\begin{align*}
-2\sqrt{2\pi}\  [T_4, \F_1^* \hat A\, \F_1] u & = 2\, \hat\gamma_1  \ast (s \eta\,
(\hat\gamma \ast u)) -2\, \hat\gamma_1 \ast ( \eta\, (  \hat\gamma_1 \ast u)) \\
& \quad
+2\,  \hat\gamma \ast (s^2 \eta\, (\hat\gamma\ast u)) -\hat\gamma\ast (s \eta\,  ( \hat\gamma_1 \ast u)).
\end{align*}
By the Young inequality,
$$
\|  [T_4, \F_1^* \hat A\, \F_1] u\|_2 \, \leq \,
C_4\, \big( \| \hat\gamma_1 \|_1 \, \|\hat\gamma\|_1 \,
\|s \eta \|_\infty + \|\hat\gamma_1\|_1^2\, \|\eta \|_\infty +
\|\hat\gamma\|_1^2 \, \| s^2 \eta\|_\infty\ \big)\, \|u\|_2,
$$
with some $C_4 < \infty$. The same argument applies to $ [T_3, \F_1^* \hat A\, \F_1]$.
We thus conclude that the first term in \eqref{double-c-1d} defines a bounded operator in ${\rm L}^2(\re)$. The
same arguments  apply to the second term in
\eqref{double-c-1d} replacing $\eta$ by $\eps'$.
As for the last term in \eqref{double-c-1d}, integration by parts shows that
\begin{align*}
& -2\sqrt{2\pi}\  [\eps \partial_s , \F_1^* \hat A\, \F_1] u  =
\eps (s) \int_{\re} \hat\gamma'(s-s') s' u(s') ds' + \eps(s)
 \int_{\re} \hat\gamma(s-s')  u(s') ds'  \\
& \qquad + \eps(s) s \int_{\re} \hat\gamma'(s-s')  u(s') ds'
 +
 \int_{\re} \big[ \hat\gamma(s-s')( \eps(s') +s'\eps'(s') -\hat\gamma'(s-s') s' \eps(s')) \big] u(s') ds' \\
 & \qquad+ s \int_{\re} \big[ \hat\gamma(s-s')  \eps'(s') - \hat\gamma'(s-s')  \eps(s') \big]  u(s') ds' .
\end{align*}
Note that each term on the right hand side of the above equation
is of the same type as one of the terms that we have already treated
above, with $\hat\gamma$ replaced by $\hat\gamma'$ when necessary.
Since $r^j\, \hat\gamma' \in \mathcal{S} (\re)$ for all $j\in\N$,
by following the same line of arguments as above we obtain
\begin{align*}
\big\| \big [ [\epsilon\, \partial_s, \F_1^*\,  \hat A\,
\F_1] ,\, \F_1^* \hat A\,  \F_1 \big]\, u\big\|_2 & \leq \tilde C\, \big( \|\eps\, (1+s^2)\|_\infty+\| \eps' (1+s^2)\|_\infty\big) \, \|u\|_2.
\end{align*}
for some constant $\tilde C< \infty$.
This together with the previous estimates implies that \eqref{double-c-1d} defines a bounded operator in ${\rm L}^2(\re)$.
\end{proof}

\noindent With these prerequisites, we can finally state the
result about the nature of the essential spectrum of
$H_{\theta'}$:

\begin{corl} \label{local}
Let $\eps$ satisfy the assumptions of Theorem \ref{main}. Let $E\in
\re\setminus {\cal E}$ be given and define the interval
$\mathcal{I}(E,\delta)=(E-\delta/2,E+\delta/2)$ as in Theorem
\ref{mourre-estim}. Then:

\begin{itemize}
\item[(a)]  $\mathcal{I}(E,\delta)$ contains at most finitely many
eigenvalues of $H_{\theta'}$, each having finite multiplicity;

\smallskip

\item[(b)] The point spectrum of $H_{\theta'}$ has no accumulation
point in $\mathcal{I}(E,\delta)$;

\smallskip

\item[(c)] $H_{\theta'}$ has no singular continuous spectrum in $\mathcal{I}(E,\delta)$.
\end{itemize}
\end{corl}

\begin{proof}
Since $A$ is self-adjoint, the statement follows from  Proposition
\ref{conditions}, Lemma \ref{double-com}, Theorem
\ref{mourre-estim} and \cite[Theorem 1.2]{pss}.
\end{proof}
\begin{proof}[Proof of Theorem \ref{main}]
Let $J\subset \re\setminus\E$ be a compact interval. For each
$E\in J$ choose $\mathcal{I}(E,\delta)$ as in Theorem \ref{mourre-estim}.
Then $J \subset \cup_{E\in J} \mathcal{I}(E,\delta)$ and since $J$ is
compact, there exists a finite subcovering:
\begin{equation} \label{cover}
J\subset \cup_{n=1}^N\, \mathcal{I}(E_n,\delta_n).
\end{equation}
By Corollary \ref{local}(a), each interval
$\mathcal{I}(E_n,\delta_n)$ contains at most finitely many eigenvalues of
$H_{\theta'}$, each of them having finite multiplicity. This
proves assertion (a). Part (b) follows immediately from (a). To
prove (c) assume that $\sigma_{\rm sc}(H_{\theta'}) \cap
(\re\setminus\E) \neq \emptyset$. Since the set $\E$ is locally
finite, see Lemma \ref{loc-finite}, it follows that there exists a
compact interval
$J\subset \re\setminus\E$ such that $\sigma_{\rm sc}(H_{\theta'}) \cap J \neq \emptyset$.
This is in contradiction with \eqref{cover} and Corollary \ref{local}(c).
Hence, $\sigma_{\rm sc}(H_{\theta'}) \cap (\re\setminus\E) = \emptyset$.
Since $\E$ is discrete, this  implies that $\sigma_{\rm sc}(H_{\theta'}) = \emptyset$. \end{proof}

\bigskip

\noindent {\bf Acknowledgements.} The authors were partially
supported  by the Bernoulli Center, EPFL, Lausanne, within
the framework of the Program ``{\em Spectral and Dynamical
Properties of Quantum Hamiltonians}", January -- June 2010.
 H. Kova\v{r}\'{i}k gratefully acknowledges also partial support of the MIUR-PRINÕ08
 grant for the project ``{\em Trasporto ottimo di massa, disuguaglianze geometriche e
 funzionali e applicazioni}'',   and would like to thank as well the Faculty of
 Mathematics of Pontificia Universidad Cat\'olica de Chile for the warm hospitality extended to him.
 G. Raikov was partially supported by by the Chilean Science Foundation {\em Fondecyt} under
Grant 1130591, and by {\em
N\'ucleo Cient\'ifico ICM} P07-027-F.


\bigskip

\newpage

{\sc Ph. Briet}\\
Universit\'e du Sud Toulon Var\\
Centre de Physique Th\'eorique\\
CNRS-Luminy, Case 907\\
13288 Marseille, France\\
E-mail: briet@cpt.univ-mrs.fr\\

{\sc H. Kova\v r\'{\i}k}\\
Dipartimento di Matematica\\
Universit\`a degli studi di Brescia\\
Via Branze, 38 \\
25123  Brescia, Italy\\
E-mail: hynek.kovarik@ing.unibs.it \\

{\sc G. Raikov}\\
Facultad de Matem\'aticas\\
Pontificia Universidad Cat\'olica de Chile\\
Av. Vicu\~na Mackenna 4860\\
 Santiago de Chile\\
E-mail: graikov@mat.puc.cl\\


\end{document}